\documentclass[12pt]{article}

\usepackage{amssymb}
\usepackage{amsthm}
\usepackage{amsmath} %txfonts mathtools
\usepackage{color}
\usepackage[all,web]{xy}
\usepackage{float}
\usepackage{pifont}
\usepackage{tikz}
\usepackage{pgfplots}
\usetikzlibrary{patterns}
\usetikzlibrary{arrows}
 \usepackage{tensor}
\usepackage{stmaryrd,bm}
 \usepackage{soul}
 \usepackage{cancel}
\usepackage{lineno}
\usepackage{mathrsfs}
\usepackage[linkcolor=blue,colorlinks=true]{hyperref}

\newcommand{\ei}{\varepsilon}

\newcommand{\tb}{\bm \tau}

\newcommand{\om}{\omega}
\newcommand{\Om}{\Omega}

\newcommand{\br}{\bm r}
\newcommand{\vr}{\varrho}
\newcommand{\vri}{\vr_{-}}
\newcommand{\vre}{\vr_{+}}

\newcommand{\bk}{{\bm k}}

\newcommand{\de}{\partial}

\newcommand{\bet}{\hat{\bm k}}

\newcommand{\rf}[1]{(\ref{#1})}
\newcommand{\dst}{\displaystyle}

\newcommand{\nonu}{\nonumber}

\renewcommand{\leq}{\leqslant}
\renewcommand{\geq}{\geqslant}
\newcommand{\s}{\Pi}

\newcommand{\gi}{\gamma_{-}}
\renewcommand{\ge}{\gamma_{+}}

\newcommand{\cB}{\mathcal B}

\renewcommand{\nu}{\om}

\newcommand{\N}{\mathcal N}
\newcommand{\Np}{{\mathcal N}^{+}}
\newcommand{\Nm}{{\mathcal N}^{-}}
\newcommand{\up}{u^{+}}

\newcommand{\oh}{\frac{1}{2}}
\newcommand{\bkp}{{\bk_{+}}}
\newcommand{\kp}{{k_{+}}}
\newcommand{\km}{{k_{-}}}
\newcommand{\ph}{\hat{\psi}}
\newcommand{\f}{\textrm{ph}}
\renewcommand{\setminus}{\smallsetminus}

\makeatletter
\newsavebox{\@brx}
\newcommand{\llangle}[1][]{\savebox{\@brx}{\(\m@th{#1\langle}\)}%
  \mathopen{\copy\@brx\mkern2mu\kern-0.9\wd\@brx\usebox{\@brx}}}
\newcommand{\rrangle}[1][]{\savebox{\@brx}{\(\m@th{#1\rangle}\)}%
  \mathclose{\copy\@brx\mkern2mu\kern-0.9\wd\@brx\usebox{\@brx}}}
\makeatother

\addtocounter{figure}{0}
\newtheorem{theorem}{Theorem}

\newtheorem{lemma}{Lemma}

\theoremstyle{remark}

\newcommand{\I}{\ensuremath{\mathrm{i}}}
\newcommand{\E}{\ensuremath{\mathrm{e}}}
\newcommand{\D}{\ensuremath{\mathrm{d}}}

\begin{document}

% \markboth{\hfill \today \hspace{1ex}  \hfill}{\hfill \today \hspace{1ex} \hfill}

%%%% Article title to be placed here
\title{Dispersion of waves in two and three-dimensional periodic media\\}
\author{Yuri A. Godin\thanks{Department of Mathematics and Statistics, University of North Carolina at Charlotte,
Charlotte, NC 28223 USA. E-mail: ygodin@uncc.edu}~ and Boris Vainberg\thanks{Department of Mathematics and Statistics,
University of North Carolina at Charlotte, Charlotte, NC 28223 USA. Email: brvainbe@uncc.edu}}

\date{\today\hspace{1ex}}

\maketitle
%%%% Abstract text to be placed here %%%%%%%%%%%%
\begin{abstract}
We consider the propagation of acoustic time-harmonic waves in a homogeneous media containing periodic lattices of spherical or cylindrical inclusions. It is assumed that the wavelength has the order of the periods of the lattice while  the radius $a$ of inclusions is small.  
A new approach is suggested to derive the complete asymptotic expansions of the dispersion relations in two and three-dimensional cases as $a \to 0$ and evaluate explicitly several first terms. Our method is based on the reduction of the original singularly perturbed (by inclusions) problem to the regular one. The Neumann, Dirichlet and transmission boundary conditions are considered. The effective wave speed is obtained as a function of the wave frequency, the filling fraction of the inclusions, and the physical properties of the constituents of the mixture. Dependence of asymptotic formulas obtained in the paper on geometric and material parameters is illustrated by graphs. 
\end{abstract}
%%%%%%%%%%%%%%%%%%%%%%%%%%%

\section{Introduction}
\setcounter{equation}{0}

Periodic media offer a great deal of possibilities for manipulating wave propagation. These include electromagnetic waves in photonic crystals, where one can create bands and gaps in the wave spectrum, positive or negative group velocity \cite{Frandsen:06}, slowing down considerably the speed of light \cite{FV:06,MV:04,Krauss:07}, nonreciprocal media \cite{FV:01}, the self-collimation effect \cite{Witzens:02} and more that lead to the development of new devices \cite{Smith:08}.

Similar phenomena can be observed in phononic crystals \cite{Wu:04} for elastic or acoustic waves \cite{Page:04} whose band gap structure is employed in sound filters, transducer design and acoustic mirrors.  Research on breaking time-reversal symmetry in wave phenomena is a growing area of interest in the field of phononic crystals and metamaterials aimed at realizing one-way propagation devices which have many potential technological applications \cite{Nassar:17}.

Deriving an explicit dispersion relation for the Floquet-Bloch waves in two and three-dimensional periodic media is an arduous problem and is usually performed numerically \cite{JJWM:11}. However, assuming that the wavelength is long compared to the period of the lattice or a characteristic size of the scatters one can obtain an asymptotic approximation. Typically it employs the method of matched asymptotic expansions as in \cite{McIver:06, McIver:2009} for small  scatterers with the Dirichlet or Neumann boundary conditions. The latter results were further developed in \cite{Craster:17}. A semi-analytical approach using the multipole expansion method is described in \cite{MMP:02}. A rigorous analysis of a sub-wavelength plasmonic crystal was presented in \cite{Lipton:2010}, where a solution of
a nonlinear eigenvalue problem is given in terms of convergent high-contrast power series for the electromagnetic fields and the first branch
of the dispersion relation. Explicit formulas for the effective dielectric tensor and the dispersion relation are obtained in \cite{GV:19} assuming that the cell size is small compared to the wavelength, but large compared to the size of the inclusions. Some other approached are presented in \cite{Zalipaev:02,Craster:10,Vanel:17,Cherednichenko:06,Smyshlyaev:08,Joyce:17}.

We consider the propagation of waves, governed by the Helmholtz equation in a medium containing periodic lattices of spherical or cylindrical inclusions of radius $a$. We assume transmission boundary conditions on the inclusions' interface and that $a$ is small relative to the period of the lattice while the wave length is comparable to (or larger than) the lattice size. The Dirichlet and Neumann boundary conditions are also discussed. We suggest a new method for determining the dispersion relations of the Floquet-Bloch waves. The method reduces the original singularly perturbed problem to the regular one and provides explicit formulas for the dispersion relations  
in  two and three dimensional settings with rigorous estimates of the remainders. Our approach can be extended if small inclusions have arbitrary shape. The coefficients of asymptotic expansion in this case will be expressed through solutions of a regular boundary value problem. This will be published elsewhere.

\section{Formulation of the problem}
\setcounter{equation}{0}

We consider the propagation of acoustic waves through an infinite medium containing a periodic array $\cB$
of spherical obstacles. The periods of the lattice $\tb_1$, $\tb_2$ and $\tb_3$ are
normalized in such a way that $\ell = \min \{|\tb_1|, |\tb_2|, |\tb_3|\}=1$, while the radius
of the balls $r=a \ll 1$  (see Figure 1).

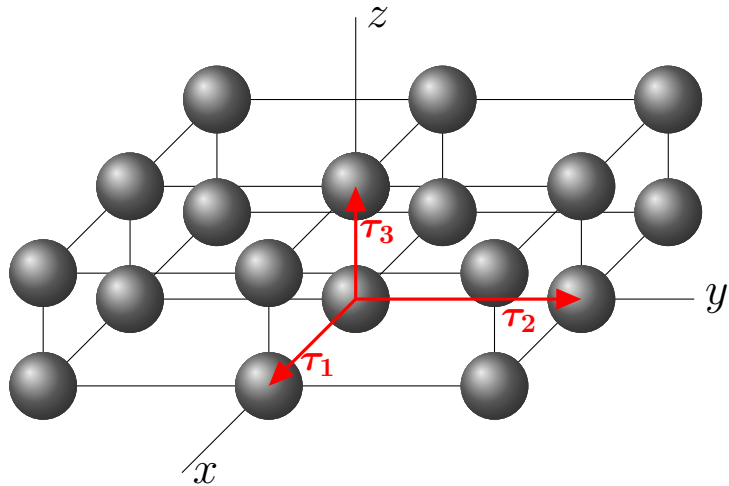
\begin{figure}[th]
\begin{center}
\begin{tikzpicture}[scale=1.5,>=triangle 45]

% coordinate system
\begin{scope}[>=latex]
\draw (0,0) -- (3,0,0) node [right] {\Large  $y$};
\draw (0,0) -- (0,2.5,0) node [right] {\Large  $z$};
\draw (0,0) -- (0,0,4) node [right] {\Large  $x$};
\end{scope}

\draw[thin] (0,0,0) -- (-2,0,0);
\draw[thin] (0,0,0) -- (0,0,-2);
\draw[thin] (2,0,2) -- (-2,0,2) -- (-2,0,-2) -- (2,0,-2) -- (2,0,2);
\draw[thin] (0,1,0) -- (-2,1,0);
\draw[thin] (0,1,0) -- (0,1,-2);
\draw[thin] (0,1,0) -- (2,1,0);
\draw[thin] (0,1,0) -- (0,1,2);
\draw[thin] (2,1,2) -- (-2,1,2) -- (-2,1,-2) -- (2,1,-2) -- (2,1,2);

\foreach \x in {-2,0, 2}
\foreach \y in {-2,0,2}
\draw[thin] (\y,0,\x) -- (\y,1,\x);

\shade[ball color=gray] (0,0,0) circle (3mm);
\shade[ball color=gray] (0,0,2) circle (3mm);
\shade[ball color=gray] (2,0,0) circle (3mm);
\shade[ball color=gray] (0,1,0) circle (3mm);
\shade[ball color=gray] (-2,0,0) circle (3mm);
\shade[ball color=gray] (-2,0,2) circle (3mm);
\shade[ball color=gray] (2,0,2) circle (3mm);
\shade[ball color=gray] (-2,0,-2) circle (3mm);
\shade[ball color=gray] (0,0,-2) circle (3mm);
\shade[ball color=gray] (2,0,-2) circle (3mm);
\shade[ball color=gray] (0,1,2) circle (3mm);
\shade[ball color=gray] (2,1,0) circle (3mm);
\shade[ball color=gray] (0,1,0) circle (3mm);
\shade[ball color=gray] (-2,1,0) circle (3mm);
\shade[ball color=gray] (-2,1,2) circle (3mm);
\shade[ball color=gray] (2,1,2) circle (3mm);
\shade[ball color=gray] (-2,1,-2) circle (3mm);
\shade[ball color=gray] (0,1,-2) circle (3mm);
\shade[ball color=gray] (2,1,-2) circle (3mm);

\coordinate [label=left:\textcolor{red}{\large \boldmath{$\tau_1$}}] (A) at (0.5,0,1.5);
\coordinate [label=left:\textcolor{red}{\large \boldmath{$\tau_2$}}] (B) at (1.9,0,0.5);
\coordinate [label=left:\textcolor{red}{\large \boldmath{$\tau_3$}}] (C) at (0.45,0.6,-0.0);

\draw [->,very thick,color=red] (0,0) -- (0,0,2);
\draw [->,very thick,color=red] (0,0) -- (2,0,0);
\draw [->,very thick,color=red] (0,0) -- (0,1,0);

\end{tikzpicture}
\end{center}
\caption{A fragment of periodic lattice of spherical inclusions of radius $a$ generated by vectors  \boldmath{$\tau_1$},
 \boldmath{$\tau_2$}, and \boldmath{$\tau_3$}. }
\label{lattice}
\end{figure}

Assuming the excess pressure $p(\br, t)$ to be time-harmonic $\dst p(\br, t) = u(\br) \E^{-\I \omega t},$
the amplitude $u(\br)$ should satisfy the equation
\begin{equation}
 \nabla \cdot \frac{1}{\vr(\br)}\, \nabla u + \gamma(\br)\, \omega^2  u = 0,
 \quad \br \notin \de \cB,
 \label{eq1}
\end{equation}
where $\vr(\br)$ is the mass density and $\gamma(\br)$ is the adiabatic bulk compressibility modulus
of the media. Both $\vr(\br)$ and $\gamma(\br)$ are periodic piecewise constant functions with the periods of the lattice. In what follows, we denote functions in the
inclusions and the external medium using the subscripts $-$ and $+$, respectively. Thus,
\begin{equation}
 \Delta u + k_{\pm}^2  u = 0,
 \quad \br \notin \de \cB,
 \label{eq}
\end{equation}
where
\begin{align}
\label{km}
  \km = & \,\sqrt{\vri \gi}\, \om, \quad \br \in \cB, \\[2mm]
  \kp = & \,\sqrt{\vre \ge}\, \om, \quad \br \notin \cB.
  \label{kp}
\end{align}

We suppose that inclusions are penetrable and therefore impose the transmission conditions
on their boundaries
\begin{align}
\label{bc1}
\left. \left\llbracket u (\br) \right\rrbracket \right . &=0, \\[2mm]
\left. \left\llbracket \frac{1}{\vr(\br)} \frac{\de u (\br)}{\de n} \right\rrbracket \right. &=0,
\label{bc2}
\end{align}
where the brackets $\llbracket \cdot \rrbracket$ denote the jump  of the enclosed quantity across
the interface $\de \cB$ of the inclusions.
Solution of \rf{eq} is sought in the form of Floquet-Bloch waves
\begin{equation}
 u(\br) = \Phi (\br) \,\E^{-\I \bk \cdot \br},
\end{equation}
where $\bk = (k_1,k_2,k_3)$ is the wave vector and $\Phi (\br)$ is a periodic function
with the periods of the lattice. The latter condition implies that  function
$\E^{-\I\bk \cdot \br} u(\br)$ is periodic. We write this as
\begin{equation}
 \rrbracket  \E^{\I\bk \cdot \br} u(\br) \llbracket =0,
\label{FB1}
\end{equation}
where the inverted brackets $\rrbracket  \cdot \llbracket$ denote the jump of the enclosed expression and their first derivatives
across the opposite sides of the cells of periodicity.

We reduce the above problem to the fundamental cell $\s$ centered at the origin:
\begin{equation}
\begin{array}{l}
\Delta u + k^2_{-}  u = 0, \quad r < a, \\[2mm]
\Delta u + k^2_{+}  u = 0, \quad \br \in \s \cap \{r>a\},
\end{array}
\label{Hz1}
\end{equation}
\begin{align}
\label{bc1a}
\left\llbracket u (\br) \right\rrbracket =0, \quad
\left\llbracket \frac{1}{\vr(\br)} \frac{\de u (\br)}{\de n} \right\rrbracket =0, \quad
\rrbracket  \E^{\I\bk \cdot \br} u(\br) \llbracket =0.
\end{align}

In the inclusionless case there is a simple dispersion relation between the time frequency $\om$ and the spacial frequency $\bk$. Namely, there are waves propagating in any direction and $\om=c|\bk|$, where $c=1/\sqrt{\ge\vre}$ is the speed of waves in the host medium. The dispersion relation $\om=H(\bk,a)$ in the presence of inclusions is more complicated and our goal is to find it when $a$ is small. We will write the dispersion relation in the form $|\bk|^2=G(\bet, \om,a),$ where $\bet=\bk/|\bk|,$

Interaction of the incident wave with inclusions creates multiple Floquet-Bloch waves with the same direction. We consider the waves for which $ |\bk|$ is close to $ \om/c$ when $a$ is small. For other waves, $ |\bk+{\bm L}|$ is close to $ \om/c$, where ${\bm L}$ is an arbitrary period of the dual lattice. These waves can be studied similarly.  

We consider both three and two dimensional cases ($d=3,2$). If $d=2$, the inclusions have the shape of infinite cylinders shown in Figure \ref{cylinders}.

\begin{figure}[th]
\begin{center}
\begin{tikzpicture}[scale=0.20]
\begin{scope}[>=latex,scale=3]
\draw (0,0,4.5) -- (5.5,0,4.5) node [right] {\Large  $y$};
\draw (0,0,0) -- (0,3,0) node [right] {\Large  $z$};
\draw (2,0,0) -- (2,1,14) node [below] {\Large  $x$};
\end{scope}
  \foreach \x in {-2,...,2}{
  \foreach \y in {-2,...,2}{
\begin{scope}[shift={(\x*45mm + \y*30mm,\y*30mm)}]
%    \draw [fill=orange!20, opacity=1.0]
  \draw [fill=gray, fill opacity=.25]
  (180:5mm) coordinate (a)
  -- ++(0,-60mm) coordinate (b)
  arc (180:360:5mm and 1.75mm) coordinate (d)
  -- (a -| d) coordinate (c) arc (0:180:5mm and 1.75mm);
  \draw [fill=gray, fill opacity=.25]
  (0,0) coordinate (t) circle (5mm and 1.75mm);
\end{scope}
}}

\end{tikzpicture}
\end{center}
\caption{A fragment of periodic lattice of infinite cylinders of radius $a$.}
\label{cylinders}
\end{figure}
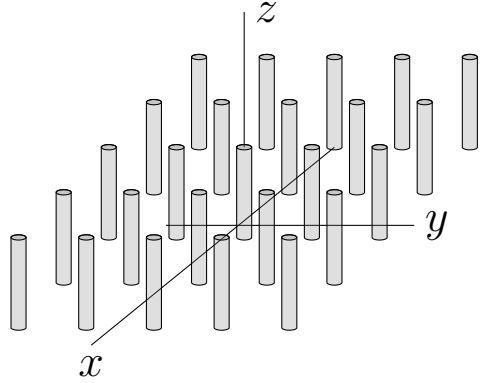

\begin{theorem}\label{t1}
 Let $d=3$. Then function $G$ in the dispersion equation 
 \begin{align}\label{dis1}
 |\bk|^2 =G(\bet,\om,a), \quad \om\neq 0,
\end{align}
is infinitely smooth. Moreover,
% Solving for $|\bk|$ we get
\begin{align}\label{dis2}
 |\bk|^2 &= \frac{\om^2}{c^2} \left(1 + c_1 a^3 + c_2 a^5 \right) +g(\hat{\bk},\om, a), \quad |g|<C(\om)a^6, \quad a\to 0,
\end{align}
where \rf{dis2}  has the following form in the case of the Neumann condition on inclusion boundaries 
\begin{equation}\label{n3}
  |\bk|^2 = \frac{\om^2}{c^2}\left(1+ \oh f+\frac{3}{20}
  \left(\frac{\om a}{c} \right)^2 f  \right) + O\left(a^6 \right), \quad f=\frac{4\pi a^3}{3|\Pi|}.
 \end{equation}
 Here $c=1/\sqrt{\gamma_+\vre}$  is the speed of waves in the host medium, $|\Pi|$ is the volume of the cell $\Pi$ and $f$ is the filling fraction of the inclusions.
 
 In the case of transmission conditions, we have
\begin{align}
 \label{c1}
 c_1 a^3 &= \left(\frac{\vri - 4\vre}{\vre + 2\vri}+ \frac{\gi}{\ge} \right)f, \\[2mm]
 \label{c2}
 c_2 a^5 &= \frac{1}{15}\left(\frac{\om a}{c} \right)^2 f
 \left[\left(1- \frac{\gi}{\ge} \right)\left(9- \frac{5\gi}{\ge} \right)-\frac{\gi}{\ge}\left(1- \frac{\gi}{\ge} \frac{\vri}{\vre}\right) \right] \nonu \\[2mm]
 &+\frac{9}{5}\left(\frac{\om a}{c} \right)^2
 \frac{f}{(\vre + 2\vri)^2} \left[ \vre^2 - \vri^2 -\vri\vre \left(1- \frac{\gi}{\ge} \frac{\vri}{\vre}\right) \right].
\end{align}
\end{theorem}

\begin{theorem}\label{t2}
 Let $d=2$. Then the dispersion relation has the form 
\[
 |\bk|^2 =\sum_{i=0}^\infty \sum_{j=0}^\infty \sigma_{ij}(\hat{\bk},\om) a^{2i} \left(a^2 \ln a\right)^{j}, \quad a\to 0.
\]
Moreover,
\begin{align}
\label{disp2d}
 |\bk|^2 = \frac{\om^2}{c^2} \left( 1 +c_1a^2-c_2a^4\ln \frac{\om a}{c}\right) + g(\hat{\bk},\om, a), \quad |g|<C(\om)a^4,
\end{align}
where \rf{disp2d}  has the following form in the case of the Neumann condition on inclusion boundaries 
\begin{equation}\label{n2}
  |\bk|^2 = \frac{\om^2}{c^2}\left(1+ f-\frac{3}{2}
  \left(\frac{\om a}{c} \right)^2 f \ln \frac{\om a}{c} \right) + O\left(f^2 \right), \quad f=\frac{\pi a^2}{|\Pi|}.
 \end{equation}
 In the case of transmission conditions, we have
\begin{align}\label{t2a}
 c_1a^2 = \left(2\alpha-1+ \frac{\gi}{\ge}\right)f , \quad ~~ c_2a^4=\left[\oh \left(1- \frac{\gi}{\ge}\right)^2 + \alpha^2 \right] \left( \frac{\om a}{c}\right)^2 f,
\end{align}
where  $\dst \alpha = \frac{\vri -\vre}{\vre + \vri}$.
\end{theorem}

Note that the Neumann condition on inclusion boundaries can be obtained from the transmission conditions by passing to the limit $\lim\vri\to \infty, \lim\gi\vri\to 0$. While formulas \rf{n3}, \rf{n2} coincide with \rf{c1}, \rf{t2a}, respectively, when $\vri^{-1}=\gi=0$, we do not justify these limiting transitions in the formulas. Instead, one can obtain the results for Neumann boundary conditions independently using the same approach.

Our approach can be also applied to the problem with the Dirichlet boundary condition. The answer in this case differs considerably from the transmission problem. Namely,
\begin{itemize}
\item 
 $d=3$, Dirichlet boundary condition
 \begin{equation}\label{d3}
  |\bk|^2 = \frac{\om^2}{c^2} -\frac{4\pi a}{|\Pi|}
  + O\left(a^2\right).
 \end{equation}
 \item 
 $d=2$, Dirichlet boundary condition
 \begin{equation}\label{d2}
  |\bk|^2 = \frac{\om^2}{c^2} -\frac{2\pi}{\dst |\Pi|\ln \frac{\om a}{c}}
  + O\left(\ln^{-2} \left(\frac{\om a}{c} \right)  \right).
 \end{equation}
\end{itemize}

\section{Outline of the proof}
\setcounter{equation}{0}

We begin with the observation that the unperturbed eigenvalue problem
\begin{equation}
 \left\{
 \begin{array}{rl}
  (\Delta+k^2_{+})u(\br) &= \lambda u(\br), \\[2mm]
  \left\rrbracket  \E^{\I\bk \cdot \br} u(\br) \right\llbracket &=0
 \end{array}
 \right.
 \label{m-egv}
\end{equation}
in the cell of periodicity $\Pi$ and a fixed $\bk$ with $|\bk|=\kp$ has a simple eigenvalue $\lambda = 0$ with the eigenfunction $u=\E^{-\I\bk \cdot \br}$. This statement is independent of whether we solve \rf{m-egv} in the Sobolev space $H^2(\Pi)$ or in the space of infinitely smooth functions, since the ellipticity of the problem implies that solutions of \rf{m-egv} are infinitely smooth. The simplicity of a zero eigenvalue can be easily established using the substitution $\E^{-\I\bk \cdot \br}u=v$ and expanding the periodic function $v$ into Fourier series.

Problem
\begin{equation}
 \left\{
 \begin{array}{rl}
  (\Delta+k^2_{\pm})u(\br) &= \lambda u(\br), \quad \br\in \Pi, \quad r\gtrless a, \\[2mm]
 \left\llbracket u (\br) \right\rrbracket \ &= \dst
\left\llbracket \frac{1}{\vr(\br)} \frac{\de u (\br)}{\de n} \right\rrbracket=0, \quad  \left\rrbracket  \E^{\I\bk \cdot \br} u(\br) \right\llbracket =0
 \end{array}
 \right.
 \label{m-egv11}
\end{equation}
can be considered as a perturbation of problem \rf{m-egv}.
For arbitrary vector $\bk$, we denote by $\bkp$ the vector with the same directions as $\bk$ and with the magnitude $\kp$, i.e. $\dst \bkp = \kp\bet$.
If $a$ and $\varepsilon=\kp -|\bk|$ are small, then the perturbation is small from the point of view of physics. Since $\lambda=0$ is a simple eigenvalue of the unperturbed problem, it follows that the eigenvalue $\lambda=\lambda(a,\varepsilon,\hat{\bk})$ of problem \rf{m-egv11} is smooth when $a$ and $|\ei|$ are small. Hence one could find the dispersion relation by finding $\lambda(a,\varepsilon,\hat{\bk})$ and solving the equation $\lambda(a,\varepsilon,\hat{\bk})=0.$

 It is not very easy to fulfill rigorously the approach discussed above, since problem \rf{m-egv11} with $a\neq 0$ is a singular perturbation of problem \rf{m-egv}, and the standard perturbation technique can not be applied.  Thus we will consider problem 
 \rf{m-egv11} only with $\lambda=0$:
 \begin{equation}
 \left\{
 \begin{array}{rl}
  (\Delta+k^2_{\pm})u(\br) &= 0, \quad \br\in \Pi, \quad r\gtrless a, \\[2mm]
 \left\llbracket u (\br) \right\rrbracket \ &= \dst
\left\llbracket \frac{1}{\vr(\br)} \frac{\de u (\br)}{\de n} \right\rrbracket=0, \quad  \left\rrbracket  \E^{\I\bk \cdot \br} u(\br) \right\llbracket =0
 \end{array}
 \right.
 \label{m-egv1}
\end{equation}
 and will reduce it to an equivalent operator equation  
\begin{equation}\label{nn}
(N^+_{\bk}-N^-_a)\psi=0
\end{equation}
for which the standard perturbation theory is valid. Here 
 $\psi$ is a function on a sphere $r=R$ of a fixed radius $R<1$ (a circle if $d=2$), and operator $N^+_{\bk}-N^-_a$ (which will be defined below) is symmetric and Fredholm.   Moreover, $N^+_{\bk}-N^-_a$ is infinitely smooth in $\bk,a$, and  has a simple zero eigenvalue when $\bk=\bkp,~a=0$. Thus,  the dispersion relation can be found from \rf{nn} by the standard  perturbation theory.

\begin{figure}[ht]
\begin{center}
\begin{tikzpicture}[>=triangle 45,scale=2.0]
% \draw[pattern=crosshatch, hatchsep = 9pt, pattern color=blue] (0,0) rectangle (3,2);
% \draw[pattern=north west lines, step=0.5cm, thick] (0,0) rectangle (3,2);
\draw[pattern=north west lines, step=0.5cm, thick] (-0.5,0) -- (0.5,2) -- (3.5,2) -- (2.5,0) -- (-0.5,0);
\draw [fill=white!50, thick] (1.5,1) circle [radius=0.7];
\draw [thick] (1.5,1) circle [radius=0.7];
\draw[fill=orange!40, opacity=1.0,thin] (1.5,1)  circle (5pt);
% \draw [fill=gray!40, opacity=1.0,pattern=grid, thick] (1.5,1) circle [radius=0.5];
% \draw[->,thick] (1.5,1) -- (2,1);
% \draw[fill=white!50] (1.5,1.1) rectangle (1.75,1.4);
\node [above] at (1.7,1.3) {$B_R$};
\node [below] at (1.8,1) {$a$};
\draw[fill=white!50] (2.3,0.35) rectangle (2.6,0.7);
\node [below] at (2.45,0.7) {\Large $\Omega$};
\draw[thin] (-0.3,1) -- (3.3,1) node[right] {$x$};
\draw[thin] (1.5,-0.3) -- (1.5,2.3);
    \end{tikzpicture}
    \caption{Decomposition of the cell of periodicity $\
    \Pi=B_R \cup \Om$.}
    \label{OmegaB}
    \end{center}
\end{figure}
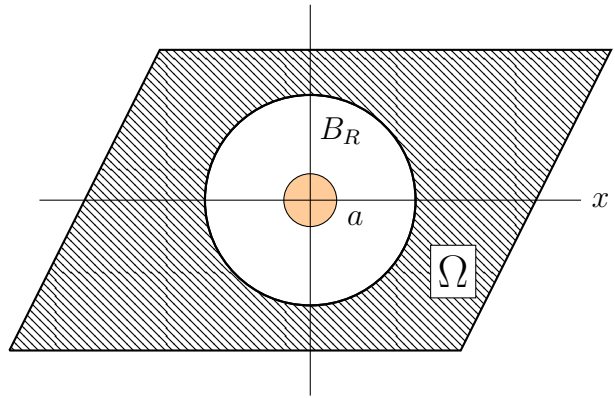

Let us define operators in \rf{nn}.  We split $\Pi$ into two domains $B_R = \{ {\br} \mid r < R\}$, such that $B_R \subset \Pi$, and its compliment $\Om = \Pi \smallsetminus B_R$ (see Figure \ref{OmegaB}). Constant $R$ will be fixed later, and $a \ll R$. Consider the following two separate problems in  $\Om $ and $ B_R$, instead of problem \rf{m-egv1}.
 \begin{align}
 \label{k-egv}
 \left( \Delta + k^2_{+} \right) \up(\br) &= 0, \quad  \br\in \Om,~~  \left\rrbracket  \E^{\I \bk \cdot \br} \up(\br) \right\llbracket =0, ~~ \left. \up\right|_{r=R}=\psi,\\
 \left( \Delta + k^2_{\pm} \right) v(\br) &= 0, \quad \br\in B_R \cap \{ r \gtrless a\}, ~~~~ \left\llbracket v (\br) \right\rrbracket =  \dst
\left\llbracket \frac{1}{\vr(\br)} \frac{\de v (\br)}{\de n} \right\rrbracket=0,~~~~ \left. v\right|_{r=R}=\psi.
 \label{uin}
\end{align}

Both problems with $\psi\in C^\infty$ have unique $C^\infty$-solutions for all values of $R$, except, possibly, a discrete set $\{R_i\}$. We fix an $R\notin \{R_i\}$ and define  operators $\Np_{\bk}$ and $\Nm_a$ as the Dirichlet-to Neumann operators with the derivatives in the direction of $r$:
\begin{equation}
\Np_{\bk},\Nm_a:H^{1}(\de B_R)\to L^{2}(\de B_R),
\quad~~ \Np_{\bk} \psi = \left.\frac{\de \up}{\de r}\right|_{r=R}, \quad  \quad
 \Nm_a \psi = \left.\frac{\de v}{\de r}\right|_{r=R},
\end{equation}
where $H^{s}(\de B_R)$ is the  Sobolev space of functions on $\de B_R$. Constant $a$ will be always assumed to be less than $R$.

\section[]{Analysis of operators $\N^+_{\bk},\N^-_{a}$}
% \section{Analysis of operators \texorpdfstring{$\N^+_{\bk},\N^-_{a}$}{TEXT}}
\setcounter{equation}{0}

\begin{lemma}\label{l2}
\leavevmode
\begin{enumerate}
 \item[(1)]
 Operators $\Np_{\bk},~\Nm_a$ and
 \begin{equation}\label{npm}
\Np_{\bk}-\Nm_a:H^{1}(\de B_R)\to L^{2}(\de B_R)
 \end{equation}
are Fredholm.
\item[(2)]
Operators $\Np_{\bk}$ and $\Nm_a$ are symmetric in $L^{2}(\de B_R)$, i.e.
 \begin{equation}\label{sym}
  \int_{r=R} \left( \Np_{\bk} \psi \right) \overline{\phi}\, \D S = \int_{r=R} \psi \, \overline{\left(\Np_{\bk} \phi\right)}\, \D S,  \quad \psi,\phi\in H^{1}(\de B_R),
 \end{equation}
and the same relation holds for $\Nm_a$.
\end{enumerate}
\end{lemma}
\begin{proof} 
Dirichlet-to-Neumann operators (for elliptic equations of the second order) are elliptic pseudo-differential operators of the first order (see \cite{vg}). Thus the first statement for operators  $\Np_{\bk},\Nm_a$ is obvious. If one proves that the difference $\Np_{\bk}-\Nm_a$ is still  elliptic (i.e., the principal symbol of the difference is not vanishing), then the first statement for \rf{npm} becomes an immediate consequence of the ellipticity. This approach could be completed, since a calculation of the symbols for Dirichlet-to-Neumann operators can be found in \cite{vg}. We prefer to prove the ellipticity of the difference by considering a simplified version $n^+-n^-$ of operator $\Np_{\bk}-\Nm_a$, where
\[
n^\pm\psi=\left.\frac{\de u^\pm}{\partial r}\right|_{r=R}
\]
are Dirichlet-to-Neumann operators related to the following problems (which are simplified versions of \rf{k-egv}, \rf{uin}):
\[
\Delta u^+=0, \quad R<r<1, \quad \left.u^+\right|_{r=R}=\psi, \quad \left.u^+\right|_{r=1}=0,
\]
\[
\Delta u^-=0, \quad r<R, \quad \left.u^-\right|_{r=R}=\psi.
\]
From local a priori estimates for elliptic operators it follows that a perturbation of the boundary value problem outside of a neighborhood of the boundary on which the Dirichlet-to-Neumann operator is defined  does not change the symbol of a Dirichlet-to-Neumann operator. The principal symbols of Dirichlet-to-Neumann operators do not depend also on the lower order terms of the equations. Thus, the principal symbols of operators $n^\pm$ are the same as the ones for operators $\Np_{\bk},\Nm_a$, respectively. Hence, the first statement of the lemma will be proved if we prove that $n^+-n^-$ is an elliptic operator of the first order. For this purpose, it is enough to show that operator
\begin{equation}\label{n+-}
n^+-n^-:H^1(\partial B_R)\to L^{2}(\partial B_R)
\end{equation}
is an isomorphism.
Consider the three-dimensional case (the two-dimensional case is similar).
Denote by $\widetilde{\Delta}$ the Laplace-Beltrami operator on the sphere $\partial B_R$. Its spectrum consists of the eigenvalues $\lambda_n=n(n+1), ~n\geq0,$ of multiplicity $2n+1$.
% , and its eigenfunctions are products of the Legendre functions of $\cos\theta$ and the exponents $e^{im\phi}$.
Let $\psi\in L^2(\partial B_R)$ and let $\psi_n$ be the projection of $\psi$ into the eigenspace of $\widetilde{\Delta}$ with the eigenvalue  $\lambda_n=n(n+1)$. Thus $\psi=\sum_{n=0}^\infty\psi_n$, and the norms of $\psi$ can be defined as follows:
\begin{equation}\label{norm1}
\|\psi\|^2_{L^2(\partial B_R)}=\|\sum_{n=0}^\infty\psi_n\|^2_{L^2(\partial B_R)},
\end{equation}
\begin{equation}\label{norm2}
\|\psi\|^2_{H^{1}(\partial B_R)}=\|\psi_0\|^2_{L^2(\partial B_R)}+\sum_{n=1}^\infty\|\widetilde{\Delta}^{1/2}\psi_n\|^2_{L^2(\partial B_R)}=
\|\psi_0\|^2_{L^2(\partial B_R)}+\sum_{n=1}^\infty n(n+1)\|\psi_n\|^2_{L^2(\partial B_R)}.
\end{equation}

Let us now estimate $(n^+-n^-)\psi$.
Obviously,
\[
u^+=\sum_{n=0}^\infty\frac{r^n-r^{-n-1}}{R^n-R^{-n-1}}\,\psi_n, \quad u^-=\sum_{n=0}^\infty\frac{r^n}{R^n}\,\psi_n,
\]
and therefore
\begin{equation}\label{dif}
(n^+-n^-)\psi=\left[\frac{\partial}{\partial r}\sum_{n=0}^\infty\frac{R^{-n-1}r^n-R^nr^{-n-1}}{R^n(R^n-R^{-n-1})}\,\psi_n\right]_{r=R}=\sum_{n=0}^\infty\frac{2n+1}{R+R^{2n+2}}\,\psi_n.
\end{equation}
Since $R<1$, the isomorphism of map \rf{n+-}  follows immediately from \rf{norm1}-\rf{dif}. This completes the proof of the first statement of the lemma.

Let us prove \rf{sym}. From the symmetry of the problem \rf{k-egv} and Green's second identity it follows that
 \[
  \int_{r=R}(u_r\overline{v}-u\overline{v}_r)\,\D S=0
 \]
 for solutions $u,v$ of \rf{k-egv} with data $\psi,\phi$ at $r=R$, respectively. The latter relation coincides with \rf{sym}.
\end{proof}

\begin{lemma}\label{l1}
\leavevmode
\begin{itemize}
\item[(1)]
Relation $\psi=u|_{r=R}$ is a one-to-one correspondence between solutions $u\in C^\infty$ of \rf{m-egv1} and solutions $\psi\in H^1(\de B_R)$ of  \rf{nn}.
\item[(2)]
Zero is a simple eigenvalue of the operator $(\Np_\bkp-\Nm_0)$ with the eigenfunction 
\begin{equation}\label{psi0}
\ph:=\E^{-\I \bkp \cdot \br}|_{r=R} = \E^{-\I \kp R \,\hat{\bk}\cdot \hat{\br}}.
\end{equation}
\end{itemize}
\end{lemma}

\begin{proof} Let $u\in C^\infty$ be a solution  of \rf{m-egv1}. Then its restrictions to $\Omega$ and $B_R$ satisfy \rf{k-egv}, \rf{uin}, respectively, with  $\psi=u|_{r=R}$. Moreover, relation  \rf{nn} holds, since $\Np_{\bk}\psi=\Nm_a\psi=u_r|_{r=R}$. Conversely, let $\psi\in H^1(\de B_R)$ satisfy  \rf{nn}. The ellipticity of the equation  \rf{nn} implies that $\psi\in C^\infty$. Let $u$ coinside with solutions of  \rf{k-egv}, \rf{uin} in  $\Omega$ and $B_R$, respectively. Then from \rf{nn} it follows that equation $(\Delta + k^2_{+})u=0$ holds also on $\de R_R$, and therefore $u$ is a solution of  \rf{m-egv1}. The first statement of the lemma is proved. The second statement follows from the first one applied to the unperturbed problem ($\bk=\bkp, a=0$) and the fact that zero is a simple eigenvalue of the problem \rf{m-egv} with the eigenfunction $\E^{-\I \bkp \cdot \br}|_{r=R}$.
\end{proof}
We will write each element $f$ in the domain  $H^{1}(\de B_R)$  and range $L^{2}(\de B_R)$ of operator \rf{npm} in the vector form $f=( f_1,f_L),$ where $f_1=c\psi_0$ is the projection of $f$ in $L^{2}(\de B_R)$ on element \rf{psi0}
and $f_L$ is orthogonal to $\ph$ in $L^{2}(\de B_R)$. Denote by $L_1,~L_0$ the subspaces of $H^{1}(\de B_R),~L^{2}(\de B_R)$, respectively, that consist of functions orthogonal in $L^{2}(\de B_R)$ to  $\psi_0$. Then, due to Lemmas \ref{l2},\ref{l1},  operator \rf{npm} has the following matrix form
\begin{equation}
 \Np_{\bkp} - \Nm_0 = \left[
 \begin{array}{cc}
  0 & 0 \\[2mm]
  0 & A
 \end{array}
 \right],
\end{equation}
where $A:L_1\to L_0$ is an isomorphism.

Obviously, operator $ \Np_{\bk}$ is infinitely smooth function of $\bk$. Thus the operator $\Np_{\bk} - \Nm_0$ in the same basis chosen for $\bk=\bkp$ has a matrix representation
\begin{equation}\label{conc}
 \Np_{\bk} - \Nm_0 = \left[
 \begin{array}{cc}
  C\ei + O(\ei^2) & O(\ei) \\[2mm]
  O(\ei) & A + O(\ei)
 \end{array}
 \right]=\left[
 \begin{array}{cc}
  C\ei + \ei^2 D_{11} & \ei D_{12} \\[2mm]
  \ei D_{21} & A + \ei D_{22}
 \end{array}
 \right], \quad \ei=\kp -|\bk|,
\end{equation}
where $C$ is a constant, $|\ei| \ll 1$ and $  D_{ij}=D_{ij}(\ei,\kp,\hat{\bk}),~\kp>0,$ are infinitely smooth functions of all the arguments. We will evaluate constant $C$ in the next Lemma.

\begin{lemma}\label{l33}
 Constant $C$ in the matrix expansion \rf{conc} of the operator
 $\Np_\bk - \Nm_0$ is equal
 \begin{equation}
 \label{C}
  C = 2\kp |\Pi|
 \end{equation}
in both dimensions $d=2$ and $d=3$.
\end{lemma}
\begin{proof}
Recall that $\ei = \kp -|\bk|$ and $\bkp=\kp\hat{\bk}$. Hence $\bk=\bkp-\ei\hat{\bk}$, and the boundary condition on $\de\Pi$ in \rf{k-egv} can be written in the form $ \left\rrbracket  \E^{\I (\bkp -\ei\hat{\bk})\cdot \br} u(\br) \right\llbracket =0$. Then problem \rf{k-egv}  in $\Omega=\Pi\setminus B_R$ becomes a regular perturbation of the same problem with $\bk=\bkp$. Thus constant $C$ in \rf{conc} is the coefficient in the linear term of the Taylor expansion of functions $q^+-q^-$ as $\ei\to 0$ (i.e., $\bk\to\bkp$), where
\begin{equation}\label{cpm}
 q^+ = \left( \Np_{\bk} \ph, \ph \right) =
 \int_{r=R} \frac{\de u^+}{\de r}\, \E^{\I \bkp \cdot \br}\, \D S, \quad q^- = \left( \Nm_0 \ph, \ph \right) =
 \int_{r=R} \frac{\de v}{\de r}\, \E^{\I \bkp \cdot \br}\, \D S.
\end{equation}
Here $\ph$  and $u^+,v$ are solutions of \rf{k-egv}, \rf{uin}, respectively,  with  $\psi=\ph$ and without the inclusion in problem \rf{uin}.

Let us evaluate $q^+$.  By the choice of $R$, problem \rf{k-egv} with $\bk=\bkp$ is uniquely solvable. If  $\psi=\ph$, then its solution is $\E^{-\I \bkp \cdot \br}$. Thus, solution $u^+$ of \rf{k-egv} can be expanded into the Taylor series in $\varepsilon$ with the zero order term being $\E^{-\I \bkp \cdot \br}$, i.e.,
\begin{equation}\label{111}
 u^+(\br)=\E^{-\I \bkp \cdot \br}+\varepsilon u_1(\br)+O(\varepsilon^2), \quad \br\in \Omega,
\end{equation}
where the remainder term is small uniformly together with all its derivatives in $\br$, and $u_1$ is the solution of the problem 
\begin{equation}\label{cpl}
 \begin{array}{l}
(\Delta  +k^2_{+} ) u_1(\br) =0, \quad \br\in \Omega, \\[2mm]
  \left\rrbracket \E^{\I \bkp \cdot \br}u_1(\br) \right\llbracket =\left\rrbracket \I \hat{\bk} \cdot \br \right\llbracket , \quad \left. u_1(\br) \right|_{r=R} = 0.
 \end{array}
\end{equation}
We note that
\begin{equation}
 \int_{r=R} \frac{\de \E^{-\I\bkp \cdot \br}}{\de r}\, \E^{\I\bkp \cdot \br}\, \D S=-\I\int_{r=R} \bkp \cdot \hat{\br}\, \D S=0.
\end{equation}
Thus from \rf{cpm} and \rf{cpl} it follows that
\begin{equation}
\label{qp}
q^+=\varepsilon\int_{r=R} \frac{\de u_1(\br)}{\de r}\, \E^{\I \bkp \cdot \br}\, \D S+O(\varepsilon^2), \quad \varepsilon\to 0.
\end{equation}
In order to evaluate the integral above, we make the substitution $u_1(\br)=w(\br)+\I( \hat{\bk} \cdot \br)\E^{-\I \bkp \cdot \br}$, which reduces \rf{cpl} to the following problem for $w$ with the homogeneous boundary condition on $\de\Pi$:
\begin{equation}\label{cplw}
 \begin{array}{l}
(\Delta  +k^2_{+} ) w(\br) =-2\kp \E^{-\I \bkp \cdot \br}, \quad \br\in \Omega, \\[2mm]
  \left\rrbracket \E^{\I \bkp \cdot \br}w(\br) \right\llbracket =0 , \quad \left. w(\br) \right|_{r=R} = -\I (\hat{\bk} \cdot \br)\E^{-\I \bkp \cdot \br}.
 \end{array}
\end{equation}
Denote the integral in \rf{qp} by $q^+_1$. Then
\[
% \label{q1p}
q^+_1=\I\int_{r=R} \frac{\de ((\hat{\bk} \cdot \br)\E^{-\I\bkp \cdot \br})}{\de r}\, \E^{\I\bkp \cdot \br}\, \D S+\int_{r=R} \frac{\de  w(\br)}{\de r}\, \E^{\I \bkp \cdot \br}\, \D S:=q^+_{11}+q^+_{12}.
\]
We have
\[
q^+_{11}=\I\int_{r=R}(\hat{\bk} \cdot \hat{\br})\, \D S+\kp R\int_{r=R}(\hat{\bk} \cdot \hat{\br})^2 \, \D S.
\]
The first integrand above is odd, and the corresponding integral is zero. Depending on the dimension $d=2,3$, the second integral above is $1/d $ times the Lebesgue measure of  $\de B_R$. Thus $\dst q^+_{11}=\frac{\kp R}{d}|\de B_R|$,
where $|\de B_R|=2\pi R$ in the dimension $d=2$ and $|\de B_R|=4\pi R^2$ if $d=3$. 

Using the Green formula and the symmetry of boundary problem \rf{cplw}, we obtain
\[
q^+_{12}=\int_{r=R} w(\br)\frac{\de\E^{\I \bkp \cdot \br}}{\de r}\, \D S-\int_{\Omega}[(\Delta  +k^2_{+} )  w(\br)]\E^{\I \bkp \cdot \br}\,d\br=\frac{\kp R}{d}|\de B_R|+2\kp |\Omega|.
\]
Hence
\begin{equation}
\label{q1p}
q^+_1=2\frac{\kp R}{d}|\de B_R|+2\kp |\Omega|=2\kp|\Pi|.
\end{equation}

Now we evaluate
\begin{equation*}
 \left( \Nm _0\E^{-\I \bkp \cdot \br}, \E^{-\I \bkp \cdot \br} \right) =
 \int_{r=R} \frac{\de v}{\de r}\, \E^{\I \bkp \cdot \br}\, \D S.
\end{equation*}
Function $v$ solves the problem
\begin{align*}
 \Delta v + k^2_{+} v &=0, \quad r<R, \\[2mm]
 \left. v \right|_{r=R} &= \E^{-\I \bkp \cdot \br}.
\end{align*}
Then
\begin{align*}
 q^- &= \int_R \frac{\de v}{\de r}\,\E^{\I \bkp \cdot \br}\, \D S =
 \int_R v\frac{\de }{\de r}\,\E^{\I \bkp \cdot \br}\, \D S \nonu \\[2mm]
 &=  \int_R \E^{-\I \bkp \cdot \br }\frac{\de }{\de r}\,\E^{\I \bkp \cdot \br}\, \D S =
 \I \kp \int_R \bkp \cdot \hat{\br}\, \D S.
\end{align*}
The latter integrand is odd, and therefore $q^-=0$. This fact together with \rf{q1p} imply \rf{C}.

\end{proof}

The next two lemmas are crucial in what follows, since they allow us to forget about the singularity of problem \rf{m-egv1} as $a\to 0$.
\begin{lemma}\label{l3}
Let $d=3$. Then the operator function $ \Nm_a-\Nm_0:H^1(B_R)\to L^2(B_R)$ is infinitely smooth in $a\geq0$, and there are bounded operators $P$ and $Q$ such that
\begin{equation}
 \label{l31}
 \Nm_a-\Nm_0=Pa^3+Qa^5+O(a^6), \quad a\to0,
\end{equation}
Moreover,
\begin{equation}
\label{l32}
\left( (\Nm_a -\Nm_0) \ph, \ph \right)  =
\frac{\om^2}{c^2 }\, |\Pi| \left(
c_1 a^3 + c_2 a^5 + O\left(a^6 \right)\right), \quad a \to 0,
\end{equation}
where $c_1$ and $c_2$ are given by \rf{c1}, \rf{c2} in Theorem \ref{t1}.
\end{lemma}

\begin{proof}We prove this lemma using explicit construction of operator $\Nm_a-\Nm_0$.
We denote the Laplace-Beltrami operator on the sphere $\de B_R$ by $\widetilde{\Delta}$. Its spectrum consists of the eigenvalues $\lambda_n=n(n+1), ~n\geq0,$ of multiplicity $2n+1$ (and its eigenfunctions are the products of the associated Legendre polynomials of $\cos\theta$ and the exponents $e^{im\phi}$). Let  $\psi_n$ be the projection in $L^2(\de B_R)$ of $\psi$ on the eigenspace of $\widetilde{\Delta}$ with the eigenvalue  $\lambda_n=n(n+1)$. Thus $\psi=\sum_{n=0}^\infty\psi_n$. Then solution of problem  \rf{uin} has the form
\begin{align}
 \label{Nma}
 u = \left\{
 \begin{array}{ll}
 \dst \sum_{n=0}^\infty a_n\, j_{n}(\km r)\,\psi_n, \quad 0 \leq r < a, \\[2mm]
 \dst \sum_{n=0}^\infty \left[ b_n\, j_{n}(\kp r)
 + c_n \,y_n (\kp r) \right]\psi_n, \quad a < r <R,
 \end{array}
 \right.
\end{align}
where $j_n,y_n$ are the spherical Bessel functions and constants $a_n,b_n,c_n$ are determined from the the boundary conditions in \rf{uin} . 
Solving the system above, we obtain that
\begin{align}
 a_n &= \frac{\kp}{\vre} \left[j_n (\kp a) y^\prime_n (\kp a)
 - j_n^\prime (\kp a) y_n (\kp a)\right]d_n^{-1} = \left( \vre \kp a^2 d_n \right)^{-1}, \\[2mm]
 b_n &= \left[\frac{\kp}{\vre}\, j_n (\km a) y^\prime_n (\kp a)
 - \frac{\km}{\vri}\, j_n^\prime (\km a) y_n (\kp a)\right] d_n^{-1},\\[2mm]
 c_n &= \left[\frac{\km}{\vri}\, j_n^\prime (\km a) j_n (\kp a)
 - \frac{\kp}{\vre}\, j_n (\km a) j_n^\prime (\kp a)\right] d_n^{-1},
\end{align}
where
\begin{align}
d_n &=
\frac{\kp}{\vre}\, j_n (\km a) \left[
j_n (\kp R) y^\prime_n (\kp a) -  j_n^\prime (\kp a) y_n (\kp R) \right] \nonu \\[2mm]
&+ \frac{\km}{\vri}\, j_n^\prime (\km a) \left[
j_n (\kp a) y_n (\kp R) -  j_n (\kp R) y_n (\kp a) \right].
\end{align}
Solution of the problem \rf{uin} without the inclusion has the form
\[
u= \sum_{n=0}^\infty\frac{ j_{n}(\kp  r)}{ j_{n}(\kp R)}\,\psi_n.
\]
 Hence, $\dst \Nm_0 \psi= \kp \sum_{n=0}^\infty\frac{ j_{n}'(\kp  R)}{ j_{n}(\kp R)}\,\psi_n$, and
\begin{align} \label{514}
 &\Nm_a \psi -\Nm_0 \psi = \kp \sum_{n=0}^\infty \left[ b_n\, j_{n}^\prime(\kp R)
 + c_n \,y_n^\prime (\kp R) -\frac{j_n^\prime (\kp R)}{j_n (\kp R)}\right]\psi_n = \sum_{n=0}^\infty \frac{f_n}{d_n\, j_n (\kp R)}\, \psi_n, %\nonu \\[2mm]
\end{align}
where
\[
 f_n = \frac{1}{\kp R^2} \left[ \frac{\km}{\vri}\, j_n (\kp a) j_n^\prime (\km a)
 - \frac{\kp}{\vre}\, j_n (\km a) j_n^\prime (\kp a) \right].
\]
From the asymptotic behavior of Bessel functions $J_n(x), Y_n(x)$ \cite{Olver:10} we have for
$0 < x \ll \sqrt{n}$ and $n\to\infty$
\begin{align}
 j_n (x)  &= \sqrt{\frac{\pi}{2x}}\, J_{n + \oh} (x) \sim \sqrt{\frac{\pi}{2x}}\, \frac{1}{\Gamma\left(n + \frac{3}{2} \right)} \left( \frac{x}{2}\right)^{n+\oh}\sim \oh \sqrt{\frac{\pi}{n}} \,\frac{1}{n!}\, \left( \frac{x}{2}\right)^{n}, \\[2mm]
 y_n (x) &= \sqrt{\frac{\pi}{2x}}\, Y_{n + \oh} (x) \sim -\sqrt{\frac{\pi}{2x}}\,
 \frac{\Gamma(n+\oh)}{\pi}\left(\frac{2}{x} \right)^{n+\oh} \sim
 -\frac{n!}{2\sqrt{\pi n}} \left(\frac{2}{x} \right)^{n+1}.
\end{align}
Then
\begin{align}
 \left| f_n \right| \leq  \,\frac{C^n_1a^{2n-1}}{(n!)^2 }, \quad
 \left| d_n \right| \geq \frac{C^n_2 }{a^2 n! }, \quad n\geq0,
\end{align}
where constants $C_i$ do not depend on $n$, and therefore $\dst \left| \frac{f_n}{d_n\, j_n (\kp R)} \right| \sim a^{2n+1}$ uniformly in $n \geq 1$ as $a\to 0$ and $\dst \left| \frac{f_0}{d_0\, j_0 (\kp R)} \right| \sim a^{3}$. The latter relations and \rf{514} justify the first statement of the lemma. 

In order to obtain the exact values of coefficients $c_1$, $c_2$ in \rf{l32} and complete the proof of Lemma \ref{l3} we need asymptotic expansions of $f_n, d_n$ for $n=0,1$ when $a \to 0 $. They look as follows
\begin{align}
d_0 &= \frac{j_0 (\kp R)}{\vre \kp a^2}\left[1+\frac{1}{6}\left(3k^2_{+} -2\,\frac{\vre}{\vri}\,k^2_{-} -k^2_{-} \right)a^2 + O\left( a^3 \right)\right], \\[2mm]
f_0 &= \frac{a}{3\kp R^2} \left[\frac{k^2_{+}}{\vre} - \frac{k^2_{-}}{\vri}
- \frac{a^2}{10} \left( \frac{\kp^4}{\vre} - \frac{\km^4}{\vri}\right)
+ \frac{a^2}{6}\,k^2_{+} k^2_{-} \left(\frac{1}{\vri} - \frac{1}{\vre} \right)\right] + O\left( a^5 \right), \\[2mm]
d_1 &= \frac{1}{3\vri \vre}\, \frac{\km}{(\kp a)^2}\, j_1 (\kp R) \left[\vre + 2\vri
-\frac{1}{10}\left( 2\vri + 3\vre \right) (\km a)^2 +\frac{1}{6}\, \vre (\kp a)^2
+ O\left( a^3 \right) \right], \\[2mm]
f_1 &= \frac{\km a}{9R^2} \left[\frac{1}{\vri} - \frac{1}{\vre} - \frac{3a^2}{10}\left(\frac{k^2_{-}}{\vri} - \frac{k^2_{+}}{\vre}\right)
+ \frac{a^2}{10}\left(\frac{k^2_{-}}{\vre} - \frac{k^2_{+}}{\vri}\right)\right] + O\left( a^5 \right).
\end{align}
Then
\begin{align}
 \Nm_a \psi -\Nm_0 \psi &= \frac{\vre a^3}{3 R^2} \left[\frac{k^2_{+}}{\vre} - \frac{k^2_{-}}{\vri} - \frac{a^2}{6}\left(\frac{k^2_{+}}{\vre} - \frac{k^2_{-}}{\vri}\right) \left(3k^2_{+} -2\,\frac{\vre}{\vri}\,k^2_{-} -k^2_{-} \right) \right. \nonu \\[2mm]
 & \left.- \frac{a^2}{10} \left( \frac{\kp^4}{\vre} - \frac{\km^4}{\vri}\right)
+ \frac{a^2}{6}\,k^2_{+} k^2_{-} \left(\frac{1}{\vri} - \frac{1}{\vre} \right)
 \right] \frac{\psi_0}{j_0^2 (\kp R)} \nonu \\[2mm]
 &+ \frac{k^2_{+} a^3}{3R^2}\, \frac{\vri \vre }{\vre + 2\vri}\left[\frac{1}{\vri} - \frac{1}{\vre} +\frac{1}{2\vri \vre}\, \frac{\vre - \vri}{\vre +2\vri}\left(\frac{1}{5}\,(2\vri + 3\vre)(\km a)^2 \right. \right. \nonu \\[2mm]
 &+\left.  \vre (\kp a)^2 \biggr)
 - \frac{3a^2}{10}\left(\frac{k^2_{-}}{\vri} - \frac{k^2_{+}}{\vre}\right)
+ \frac{a^2}{10}\left(\frac{k^2_{-}}{\vre} - \frac{k^2_{+}}{\vri}\right)\right] \frac{\psi_1}{j_1^2 (\kp R)}+ O\left( a^6 \right).
\label{diff}
\end{align}
Let us substitute $\ph$ for $\psi$ in \rf{diff}  and evaluate quadratic form \rf{l32}. 
We have
\begin{align}
 \psi_0 &= \frac{1}{4\pi R^2}\int_{r=R} \E^{-\I \bkp \cdot \br }\, \D S, \\[2mm]
 \psi_1 &=(\bet_{+} \cdot \br) \left( \int_{r=R} (\bet_{+} \cdot \br)^2\,\D S\right)^{-1}\int_{r=R} (\bet_{+} \cdot \br)\,\E^{-\I \bkp \cdot \br }\, \D S.
\end{align}
Note that
\[
 \int_{r=R} \E^{-\I \bkp \cdot \br }\, \D S = 4\pi R^2 j_0 (\kp R),
 \quad
 \int_{r=R} (\bet_{+} \cdot \br)^2\,\D S= \frac{4\pi R^4}{3}.
\]
Differentiation of the first equality above with respect to $|\bkp|$ leads to
\[
 \int_{r=R} (\bet_{+} \cdot \br)\,\E^{-\I \bkp \cdot \br }\, \D S=-4\pi \I R^3 j_1(\kp R).
\]
Therefore
\begin{align}
 \psi_0  = j_0 (\kp R), \quad 
 \psi_1 = - 3\I (\bet_{+} \cdot \br) R^{-1} j_1(\kp R),
\end{align}
and
\[
 \| \psi_0 \|^2 = 4\pi R^2 j_0^2 (\kp R), \quad
 \| \psi_1 \|^2 = 12\pi R^2 j_1^2 (\kp R),
\]
Hence
\begin{align*}
\left((\Nm_a -\Nm_0) \psi_0, \psi_0 \right)
  &= \frac{4\pi \vre a^3}{3}
 \left( \frac{k^2_{+}}{\vre} - \frac{k^2_{-}}{\vri}\right)
 -\frac{4\pi a^5}{45 \vri^2}\left(9k^4_{+}\vri^2 + k^2_{-}\vri\vre (k^2_{-}-15k^2_{+}) + 5k^4_{-}\vre^2 \right) + O\left(a^6\right), \\[2mm]
 \left((\Nm_a -\Nm) \psi_1, \psi_1 \right) &=\frac{4\pi a^3 k^2_{+} (\vre - \vri)}{\vre +2\vri}-\frac{12\pi a^5 k^2_{+}}{5(\vre+2\vri)^2}
 \left(k^2_{+}(\vre^2 -\vri^2) -\vri\vre (k^2_{+}-k^2_{-}) \right) + O\left(a^6\right).
\end{align*}
Now formula \rf{l32}  follows
from the last two relations and \rf{diff}. 
\end{proof}

\begin{lemma}\label{l5}
Let $d=2$. Then the operator function $ \Nm_a-\Nm_0:H^1(B_R)\to L^2(B_R)$ has the following asymptotic expansion as $a \to 0$
\begin{equation}
 \label{l51}
 \Nm_a-\Nm _0\sim \sum_{i=0}^\infty \sum_{j=0}^\infty N_{ij} a^{2i} \left(a^2 \ln a\right)^{j}, \quad N_{00} =0.
\end{equation}
Moreover,
\begin{equation}
\label{l52}
\left( (\Nm_a -\Nm_0) \ph, \ph \right)  =
\frac{\om^2}{c^2 }\, |\Pi| \left(
c_1 a^2 + c_2 a^4 \ln a + O\left(a^4 \right)\right), \quad a \to 0,
\end{equation}
where $c_1$ and $c_2$ are given in Theorem \ref{t2}.
\end{lemma}
\begin{proof}
 The proof is similar to that in Lemma \ref{l3}. Namely, solution $v$ of the problem  \rf{uin} is represented in the form
\begin{align}
 \label{Nma2}
 v = \left\{
 \begin{array}{ll}
 \dst \sum_{n=0}^\infty A_n\, J_{n}(\km r)\,\psi_n, \quad 0 \leq r < a, \\[2mm]
 \dst \sum_{n=0}^\infty \left[ B_n\, J_{n}(\kp r)
 + C_n \,Y_n (\kp r) \right]\psi_n, \quad a < r <R,
 \end{array}
 \right.
\end{align}
where $J_n(x),Y_n(x)$ are Bessel functions. Constants $A_n,B_n,C_n$ are determined from the boundary conditions in \rf{uin}:
\begin{align}
 A_n &= \frac{\kp}{\vre} \left[J_n (\kp a) Y^\prime_n (\kp a)
 - J_n^\prime (\kp a) Y_n (\kp a)\right]D_n^{-1} = \frac{2}{\pi \vre a D_n}, \\[2mm]
 B_n &= \left[\frac{\kp}{\vre}\, J_n (\km a) Y^\prime_n (\kp a)
 - \frac{\km}{\vri}\, J_n^\prime (\km a) Y_n (\kp a)\right] D_n^{-1},\\[2mm]
 C_n &= \left[\frac{\km}{\vri}\, J_n^\prime (\km a) J_n (\kp a)
 - \frac{\kp}{\vre}\, J_n (\km a) J_n^\prime (\kp a)\right] D_n^{-1},
\end{align}
where
\begin{align}
D_n &= Y_n (\kp R) \left[\frac{\km}{\vri}\, J_n^\prime (\km a) J_n (\kp a) 
-\frac{\kp}{\vre}\,J_n^\prime (\kp a) J_n (\km a)\right] \nonu \\[2mm]
&-J_n (\kp R) \left[\frac{\km}{\vri}\, J_n^\prime (\km a) Y_n (\kp a) 
-\frac{\kp}{\vre}\,J_n (\km a) Y_n^\prime (\kp a) \right].
\end{align}
Then
\begin{align}
\label{N2d}
 &\Nm_a \psi -\Nm _0\psi = \kp \sum_{n=0}^\infty \left[ b_n\, J_{n}^\prime(\kp R)
 + C_n \,Y_n^\prime (\kp R) -\frac{J_n^\prime (\kp R)}{J_n (\kp R)}\right]\psi_n = \sum_{n=0}^\infty \frac{F_n}{D_n\, J_n (\kp R)}\, \psi_n, 
\end{align}
where
\[
 F_n = \frac{2}{\pi R} \left[ \frac{\km}{\vri}\, J_n (\kp a) J_n^\prime (\km a)
 - \frac{\kp}{\vre}\, J_n (\km a) J_n^\prime (\kp a) \right].
\]
Using the asymptotic behavior of the Bessel function as $x \to 0$
\begin{align*}
 &J_0 (x) \sim 1, \quad &&Y_0 (x) \sim \frac{2}{\pi}\, \ln \frac{x}{2},  \\[2mm]
 &J_n (x) \sim \frac{1}{n!}\left( \frac{x}{2}, \right)^n, ~~n\geq 1, \quad &&Y_n (x) \sim - \frac{(n-1)!}{\pi}\left( \frac{2}{x} \right)^n, ~~n \geq 1, \\[2mm]
 &J^\prime_0 (x) = -J_1 (x) \sim -\frac{x}{2}, && 
\end{align*}
one can derive the form of asymptotic expansion of $D_n$ and $F_n$ as $a \to 0$
\begin{align}
\label{DnFn}
 D_n &= \sum_{j=0}^\infty \alpha_{j,n }\,a^{2j-1} + \ln (\kp a) \sum_{j=\max (n,1)} \beta_{j,n}\, a^{2j-1},\quad
 F_n = \sum_{j=n}^\infty \gamma_{j,n}\, a^{2j-1}.
\end{align}
In particular, for $n=0,1$ we have
\begin{align}
D_0 &= \frac{2}{\pi \vre a}\,J_0(\kp R)+\frac{a}{\pi}\left(\frac{k^2_{-}}{\vri}-\frac{k^2_{+}}{\vre} \right)J_0(\kp R) \ln(\kp a)+ O\left(a\right), \\[2mm]
F_0 &= \frac{a}{\pi R} \left[\frac{k^2_{+}}{\vre} - \frac{k^2_{-}}{\vri} + \frac{a^2}{8}\left(\frac{k^4_{-}}{\vri} - \frac{k^4_{+}}{\vre}  - 2k^2_{-}k^2_{+} \left(\frac{1}{\vre} - \frac{1}{\vri} \right)\right)\right] + O\left(a^5\right), \\[2mm]
D_1 &= \frac{\km}{\pi \kp a}  \left[\frac{1}{\vre} + \frac{1}{\vri}\right] J_1(\kp R) + \frac{\kp a}{2\pi}\left(\frac{1}{\vre} - \frac{1}{\vri} \right)J_1(\kp R)\ln(\kp a) + O\left(a\right), \\[2mm]
F_1 &= \frac{\kp \km a}{2\pi R} \left[\frac{1}{\vri} - \frac{1}{\vre} +\frac{a^2}{8\vri \vre} \left(k^2_{-}(\vri-3\vre) - k^2_{+}(\vre-3\vri) \right)\right] + O\left(a^5\right).
\end{align}
As a result,
\begin{align}
 \Nm_a \psi -\Nm _0\psi &= \left[\frac{\vre a^2}{2 R}\left( \frac{k^2_{+}}{\vre} - \frac{k^2_{-}}{\vri} \right)+ \frac{(k^2_{-} \vre - k^2_{+}\vri)^2}{4R \vri^2}\, a^4 \ln(\kp a)\right]
 \frac{\psi_0}{J_0^2 (\kp R)} \nonu \\[2mm]
 &+ \left[\frac{k^2_{+} a^2}{2 R}\,
 \frac{\vre - \vri}{\vre + \vri} + \frac{k^4_{+}}{4R} \left(\frac{\vre - \vri}{\vre + \vri} \right)^2  a^4 \ln(\kp a)\right]\frac{\psi_1}{J_1^2 (\kp R)} + O\left(a^4\right).
\end{align}
Substitution of \rf{DnFn} into \rf{N2d} provides first statement \rf{l51} of the Lemma.

We use the values of the integrals
\begin{align*}
 \int_{r=R} \E^{-\I \bk_{+} \cdot \br}\,\D s &= 2\pi R J_0(\kp R), \\[2mm]
 \int_{r=R} (\bet_{+} \cdot \br)\, \E^{-\I \bk_{+} \cdot \br}\,\D s &= 2\pi \I R J_1(\kp R).
\end{align*}
From here we calculate projections  $\psi_0$, $\psi_1$ of the function $\E^{-\I \bk_{+} \cdot \br}$:
\begin{align}
 \psi_0 &= \frac{1}{2\pi R}\int_{r=R} \E^{-\I \bk_{+} \cdot \br}\,\D s = J_0(\kp R), \\[2mm]
 \psi_1 &= (\bet_{+} \cdot \br) \left(\int_{r=R} (\bet_{+} \cdot \br)^2\,\D s \right)^{-1} \int_{r=R} (\bet_{+} \cdot \br)\, \E^{-\I \bk_{+} \cdot \br}\,\D s \nonu \\[2mm]
 &= \frac{\bet_{+} \cdot \br}{\pi R}\,2\pi \I R J_1(\kp R)
 =2\I (\bet_{+} \cdot \br)J_1(\kp R).
\end{align}
Hence
\begin{align}
 \| \psi_0 \|^2 = 2\pi R J_0^2 (\kp R), \quad
 \| \psi_1 \|^2 = 4\pi R J_1^2 (\kp R),
\end{align}
and therefore
\begin{align*}
 \left((\Nm_a -\Nm_0)\, \psi_0, \psi_0 \right) &=
 \vre \pi a^2\left[ \frac{k^2_{+}}{\vre} - \frac{k^2_{-}}{\vri} \right]+\frac{\pi \ln (\kp a)}{2\vri^2}\left(k^2_{-} \vre - k^2_{+} \vri \right)^2 a^4 + O\left(a^4
 \right), \\[2mm]
 \left((\Nm_a -\Nm_0)\, \psi_1, \psi_1 \right) &=
 2k^2_{+}\pi a^2\, \frac{\vre - \vri}{\vre + \vri}
 + \pi k^4_{+} \left( \frac{\vre - \vri}{\vre+\vri}\right)^2 \ln (\kp a)\, a^4 + O\left( a^4 \right). 
\end{align*}
Addition of the above formulas gives \rf{l52} in terms of notation \rf{km},\rf{kp}.
\end{proof}

\section{Proof of Theorems \ref{t1}, \ref{t2} and formulas \rf{d3}, \rf{d2}}   
\setcounter{equation}{0}

In order to prove Theorem \ref{t1}, we use Lemma \ref{l1} and reduce problem \rf{m-egv1} to the equivalent  equation \rf{nn}. Due to Lemma \ref{l2}, the operator in the latter equation is Fredholm and symmetric. Let us rewrite   \rf{nn} in the form 
\begin{equation}\label{matr}
 [(\Np_\bk-\Nm_0)+(\Nm_a-\Nm_0)]\psi=0.
\end{equation}
Obviously, the first term on the left is infinitely smooth in $\bk$ and has the matrix representation \rf{conc} in the decompositions $(\ph, L)$ of the domain and range of the operator. Recall that $\ph$ is defined by \rf{psi0} and $L$ is the subspace of the domain or range, respectively, which consists of functions orthogonal (in $L^2(\de B_R)$) to 
$\ph$. Due to Lemma \ref{l3}, the second operator in \rf{matr} is an infinitely smooth function of $a$. Hence equation \rf{matr} can be written in the form
\begin{equation}\label{matr1}
  \left[
 \begin{array}{cc}
  C\ei + \ei^2 D_{11}+B_{11}(a) & \ei D_{12}+B_{12}(a) \\[2mm]
  \ei D_{21}+B_{21}(a) & A + \ei D_{22}+B_{22}(a)
 \end{array}
 \right]
  \left(
 \begin{array}{cc}
  \phi_1 \\[2mm]
  \phi_2
 \end{array}
 \right)=0, \quad \ei=\kp -|\bk|,
\end{equation}
where $\phi_1=\sigma\ph$ is the projection of $\psi$ on $\ph$, component $\phi_2$ of $\psi$
is orthogonal to $\ph$,  
operator functions $  D_{ij}=D_{ij}(\ei,\kp,\hat{\bk}),~\kp>0,$ and $B_{1j}(a)$ are infinitely smooth,  and asymptotic expansions of operators $B_{ij}$ as $a\to 0$ are given in Lemma \ref{l3}. In particular, $\|B_{ij}\|=O(a^3),~ a\to 0$.

Since operator $A$ is invertible, the second equation of the system \rf{matr1} can be solved for $\phi_2$, and  the solution is proportional to the norm of $\phi_1$, i.e.,
\[
 \phi_2=\sigma\hat{\phi_2},
\]
where $\hat{\phi_2}$ is a $\sigma$-independent element of $L$ which is infinitely smooth in $\ei,\kp,\hat{\bk}$ and such that $\|\hat{\phi_2}\|=O(|\ei|+a^3),~|\ei|+a\to 0$. Then the first equation of \rf{matr1} implies
\[
 [C\ei+B_{11}(a)+O((|\ei|+a^3)^2)]\sigma=0.
\]
Hence, a non-trivial solution of \rf{matr1} for small $\ei,a$ exists if and only if 
\begin{equation}\label{drel}
 C\ei+B_{11}(a)+O((|\ei|+a^3)^2)=0,
\end{equation}
where the left-hand side is an infinitely smooth function of $\ei,\kp,\hat{\bk}$ (when $\kp>0)$ and operator $B_{11}(a)$ coincides with the left-hand side in \rf{l32}. The dispersion relation is defined by solving the equation above for $\ei$. We obtain that $|\bk|-\kp$ is an infinitely smooth function of $a, \kp,\hat{\bk}$ and
\[
 |\bk|-\kp=C^{-1}B_{11}+O(a^6).
\]
This justifies the infinite smoothness of the dispersion relation \rf{dis1}.
In order to justify expansion \rf{dis2}, we use \rf{l32} for $B_{11}(a)$ and formula for $C$ from Lemma \ref{l33} and arrive at
\begin{equation}\label{end}
 (|\bk|-\kp)2\kp=\frac{\om^2}{c^2 } \left(
c_1 a^3 + c_2 a^5 + O\left(a^6 \right)\right), \quad a \to 0.
\end{equation}
From here it follows that $|\bk|-\kp=O(a^3)$ and therefore $2\kp=|\bk|+\kp+O(a^3)$. The latter relation together with \rf{end} implies \rf{dis2} and completes the proof of Theorem \ref{t1}. We also note that the first correction term in \rf{dis2} agrees with that provided in \cite{Belyaev:87}.

Proof of Theorem \ref{t2} is absolutely similar. One need only to use Lemma \ref{l5} instead of Lemma \ref{l3}. Hence all the functions of $a$ will have asymptotic expansions as in  \rf{l51} instead of being infinitely smooth. Operators $B_{ij}$ will have order $O(a^2), ~a\to 0$, instead of $O(a^3)$, $B_{11}$ will coincide with the left-hand side in \rf{l52}, and \rf{drel} will have the form
\begin{equation*}
 C\ei+B_{11}(a)+O((|\ei|+a^2)^2)=0.
\end{equation*}

In the case of the Dirichlet boundary condition, norms of operators $B_{ij}$ have order $O(a)$ and $O(1/\ln a)$ when $d=3,2$, respectively, and this leads to different dispersion relations.

\section{Effective wave velocity}

Asymptotic formulas \rf{dis2}, \rf{disp2d} allow us to calculate both the effective phase velocity $\dst c_{\f} = \frac{\om}{k}$ and the group velocity $c_\ast = \dst \frac{\de \om}{\de k}$. With the accuracy of the asymptotic formulas we have

\begin{itemize}
 \item
 $d=2$, transmission boundary conditions:
 \begin{align}
  \label{2dtg}
   c_\ast &= c \left( 1 -\oh \,c_1 a^2 + \frac{3}{2}\, c_2a^4\ln \frac{\om a}{c}\right), \\[2mm]
   c_\f &= c \left( 1 -\oh \,c_1 a^2 + \oh\, c_2a^4\ln \frac{\om a}{c}\right).
   \label{2dtph}
 \end{align}
 \item
  $d=2$, Neumann boundary conditions:
 \begin{align}
 \label{2dng}
 c_\ast &= c \left(1- \frac{1}{2}\, f+ \frac{9}{4}
  \left(\frac{\om a}{c} \right)^2 f \ln \frac{\om a}{c}  \right), \\[2mm]
  c_\f &= c \left(1- \frac{1}{2}\, f+ \frac{3}{4}
  \left(\frac{\om a}{c} \right)^2 f \ln \frac{\om a}{c}  \right). 
  \label{2dnph}
 \end{align}
 \item 
 $d=3$, transmission boundary conditions:
 \begin{align}
  \label{3dtg}
  c_\ast &= c \left(1 - \oh\, c_1 a^3 - \frac{3}{2}\, c_2 a^5 \right), \\[2mm]
  c_\f &= c \left(1 - \oh\, c_1 a^3 - \oh\, c_2 a^5 \right),
  \label{3dtph}
  \end{align}
 \item
  $d=3$, Neumann boundary conditions:
\begin{align}
\label{3dng}
 c_\ast &= c \left(1- \frac{1}{4}\, f- \frac{9}{40}
  \left(\frac{\om a}{c} \right)^2 f  \right), \\[2mm]
 c_{\f} &= c \left(1- \frac{1}{4}\, f- \frac{3}{40}
  \left(\frac{\om a}{c} \right)^2 f  \right),
  \label{3dnph}
 \end{align}
\end{itemize}
Here $c_1$, $c_2$ are defined by \rf{c1},\rf{c2} in three dimensions and are given by \rf{t2a} for $d=2$. The formulas are consistent with the results reported in \cite{Krokhin:2003}, \cite{Torrent:2006}, \cite{Torrent:2007}.
If $\gi/\ge \gg 1$ as in the case of air bubbles in water where $\gi/\ge \sim 1.5 \cdot 10^4$, then even a tiny concentration of the air bubbles causes a dramatic reduction of the wave speed. This effect was confirmed numerically in \cite{Kafesaki:00} and \cite{Torrent:2006}, however it is not attributed to the Minnaert resonance or wave localization. We rewrite first correction terms \rf{c1} and \rf{t2a} in formulas
\rf{dis2} and \rf{disp2d}, respectively, in terms of the wave speed $c_\pm$ in the constituent media
\begin{align}
 c_1 a^2 &= \frac{(\vri c_{-} - \vre c_{+})^2+ \vri\vre (c_{+}-c_{-})(c_{+}+3c_{-})}{\vri c^2_{-}(\vre + \vri)} f, \\[2mm]
 c_1 a^3 &= \frac{(\vri c_{-} - \vre c_{+})^2+ 2\vri\vre (c_{+}-c_{-})(c_{+}+2c_{-})}{\vri c^2_{-}(\vre + 2\vri)} f.
\end{align}
The first terms in the numerators of the formulas indicate the mismatch in the characteristic impedances $\vr c$ of the media. 
These formulas show that if the scatterers are slow ($c_{-} < c_{+}$) then the effective wave speed will {\it always} be slower than the wave speed in the host medium. On the other hand, the presence of fast scatterers $(c_{-} > c_{+})$ does not guarantee that the effective wave velocity will be larger than that in the matrix. It will be so only when the characteristic impedances of the two media are close enough so that their mismatch does not cause strong scattering. 

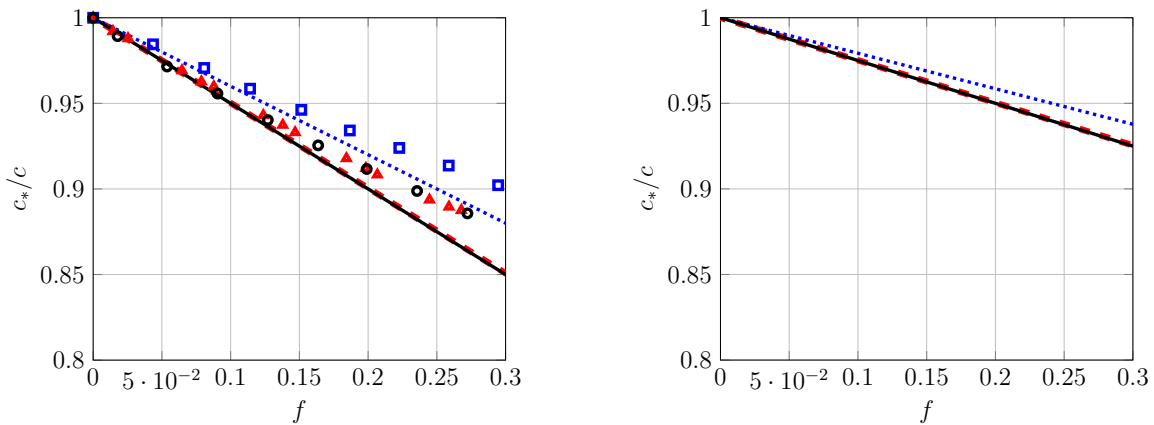
\begin{figure}[ht]
\begin{center}
\begin{minipage}{.4\textwidth}
\begin{tikzpicture}[scale=0.8,remember picture]
    \begin{axis}[
    		xlabel=$f$,
		    ylabel=$c_\ast/c$,
		    grid=major,
%         domain=0.125:220,
        xmin=0, xmax=0.3,
        ymin=0.8, ymax=1.0,
%         legend style={at={(0.5,-0.3)},
%       anchor=north,legend columns=2},
    ]
% \addplot[blue,ultra thick, domain=0.0:0.9,samples=10] {1-0.5*0.5*x};%+0.1*x*x*ln(0.01*x)}; ethanol in hydrogen
% \addplot+[ultra thick, domain=0.0:0.9,samples=10] {1-0.5*0.124e-1*x}; % kerosin in water
% \addplot+[ultra thick, domain=0.0:0.9,samples=10] {1-0.5*0.2340746747e-1*x}; % water in kerosin
% \addplot+[red, ultra thick,smooth,mark=o,mark options={scale=2}] plot coordinates{
\addplot+[red, ultra thick, only marks, mark=triangle] plot coordinates{
        (0.00000, 1.00000)
	    (0.01442, 0.99210)
	    (0.02535, 0.98792)
	    (0.06458, 0.96919)
	    (0.07873, 0.96268)
	    (0.08774, 0.95984)
	    (0.12375, 0.94300)
	    (0.13790, 0.93735)
	    (0.14691, 0.93304)
	    (0.18421, 0.91785)
	    (0.19836, 0.91211)
	    (0.20672, 0.90835)
	    (0.24466, 0.89367)
	    (0.25881, 0.88956)
	    (0.26781, 0.88753)
	    (0.30383, 0.87306)
	    (0.31798, 0.86895)
	    (0.32698, 0.86602)
	    (0.36428, 0.85253)
	    (0.37843, 0.84784)
	    (0.38679, 0.84591)
	    (0.42345, 0.83346)
	    (0.43760, 0.82896)
	    (0.44660, 0.82687)
	    (0.48262, 0.81516)
	    (0.49677, 0.81066)
	    (0.50577, 0.80924)
	    (0.54307, 0.79626)
	    (0.55722, 0.79311)
	    (0.56622, 0.79054)
	    (0.60353, 0.77612)
	    (0.61767, 0.77124)
	    (0.62604, 0.76955)
	    (0.66269, 0.75464)
	    (0.67684, 0.74976)
	    (0.68585, 0.74610)
	    (0.72315, 0.72786)
	    (0.73730, 0.72086)
	    (0.74630, 0.71716)
	    (0.78360, 0.69068)
	    (0.79775, 0.67964)
	    (0.80611, 0.67324)
	    (0.84277, 0.62282)
	    (0.85692, 0.59889)
	    (0.88450, 0.51960)
	    (0.88997, 0.49665)
	    (0.90258, 0.42371)
	    (0.90451, 0.40799)
	    }; % water in air
% \addlegendentry{water in air, []}
\addplot [red,line width= 3pt, dashed, dash pattern=on 5pt off 5pt] {1-0.5*0.995*x}; % water in air
% \addlegendentry{water in air, asympt}
\addplot [color=blue, ultra thick,only marks,mark=square] plot coordinates{
(0.00000, 1.00000)
(0.04352, 0.98449)
(0.08077, 0.97067)
(0.11414, 0.95857)
(0.15138, 0.94620)
(0.18670, 0.93416)
(0.22266, 0.92396)
(0.25863, 0.91366)
(0.29459, 0.90219)
(0.33121, 0.89316)
(0.36589, 0.88397)
(0.40251, 0.87499)
(0.43785, 0.86757)
(0.47383, 0.86014)
(0.51044, 0.85160)
(0.54514, 0.84480)
(0.58178, 0.84043)
(0.61711, 0.83203)
(0.65310, 0.82607)
(0.68973, 0.81967)
(0.72443, 0.81447)
(0.76106, 0.80786)
(0.79640, 0.80126)
(0.83237, 0.79241)
(0.86834, 0.78263)
(0.90429, 0.76837)
};
% \addlegendentry{mercury in water, []}
% mercury in water
\addplot [blue,ultra thick,dotted] {1-0.5*0.8*x}; % mercury in water
% \addlegendentry{mercury in water, asympt}
\addplot+[black,ultra thick,only marks,mark=o,draw=black,fill=black] plot coordinates{
(0.00000, 1.00000)
(0.01779, 0.98925)
(0.05377, 0.97152)
(0.09040, 0.95574)
(0.12703, 0.94023)
(0.16366, 0.92545)
(0.19898, 0.91166)
(0.23561, 0.89882)
(0.27224, 0.88571)
(0.30822, 0.87358)
(0.34419, 0.86137)
(0.38017, 0.85000)
(0.41615, 0.83814)
(0.45212, 0.82803)
(0.48810, 0.81736)
(0.52604, 0.80503)
(0.55744, 0.79528)
(0.58491, 0.78569)
(0.61238, 0.77809)
(0.63986, 0.76976)
(0.66733, 0.75967)
(0.69480, 0.74856)
(0.72227, 0.73568)
(0.74975, 0.72105)
(0.77722, 0.70164)
(0.80469, 0.67642)
(0.83086, 0.64261)
(0.85179, 0.60309)
(0.86618, 0.56426)
(0.87665, 0.52529)
(0.88506, 0.48181)
(0.89038, 0.44113)
(0.89627, 0.39644)
(0.90150, 0.34675)
(0.90265, 0.30814)
(0.90608, 0.26646)
(0.90870, 0.23368)
}; % rigid
% \addlegendentry{rigid in air, []}
\addplot [black,ultra thick] {1-0.5*x}; % rigid
% \addlegendentry{rigid in air, asympt}
\end{axis}
\end{tikzpicture}
\end{minipage}
\begin{minipage}{.4\textwidth}
\hspace*{10mm}
\begin{tikzpicture}[scale=0.8]
    \begin{axis}[
    		xlabel=$f$,
		    ylabel=$c_\ast/c$,
		    grid=major,
%         domain=0.125:220,
        xmin=0, xmax=0.3,
        ymin=0.8, ymax=1.0,
        legend style={at={(0.5,-0.3)},
      anchor=north,legend columns=-1},
    ]
\addplot+[red,line width= 3pt, dashed, dash pattern=on 5pt off 5pt] {1-0.5*0.497*x}; % water in air
% \addlegendentry{water in air}

\addplot [blue, ultra thick,dotted, smooth] {1 - 0.5*0.415*x}; % mercury in water
% \addlegendentry{mercury in water}

\addplot [black,ultra thick] {1 - 0.25*x}; % rigid
% \addlegendentry{rigid in air}

\end{axis}
\end{tikzpicture}
\end{minipage}
\caption{Long wavelength sound velocity for a hexagonal lattice of cylindrical (left) and an arbitrary three-dimensional lattice of spherical (right) inclusions as a function of the filling fraction, $f$.  Lines correspond to the asymptotic formulas \rf{2dtg}, \rf{2dng}, \rf{3dtg}, \rf{3dng} for water in air (red dashed), mercury in water (blue dotted) and rigid inclusions in air (solid black). Symbols depict results of numerical calculations for a cluster made of $151$ cylinders \cite{Torrent:2006} of water in air (red triangles), mercury in water (blue squares), and rigid cylinders (black circles). }
\label{graphs}
\end{center}
\end{figure}
%%%%%%%%%%%%%%%%%%%%%%%%%%%%%%%%%%%%%%%%%%%%%%%%%%%%%%%%%%%%%%%%%%%%%%%%%%%%%%%%%%%%%%%%%%%%%%%%%%%%%%%
Figure \ref{graphs} depicts the comparison of the asymptotic formulas of the present paper with numerical calculations of the sound velocity for a hexagonal cluster of $151$ fluid cylinders embedded in a fluid or gas (left panel) reported in \cite{Torrent:2006} when $\omega \to 0$. The right panel shows dependence of the long wavelength sound group velocity for a lattice of spherical inclusions of the same constituents calculated by formulas \rf{3dtg}, \rf{3dng}. Compared with the two-dimensional case, the sound velocity decays about twice as slowly. We note that the only geometric parameter the sound velocity depends on in long wavelength approximation is the filling fraction, $f$.

Increasing the wave frequency $\omega$ enhances the scattering and makes the decay of the sound velocity even faster as shown in Figure \ref{graphs1} for $\omega/c = 1$ and $|\Pi |=1$. This effect is also illustrated in Figure \ref{graphs2} when the group velocity of sound in an orthorhombic  lattice of rigid spherical inclusions decreases as $\omega/c$ increases from $0$ to $3$.

\begin{figure}[ht]
\begin{center}
\begin{minipage}{.4\textwidth}
\hspace*{10mm}
\begin{tikzpicture}[scale=0.8]
\def\oo{1}; \def\pp{3.1415};
    \begin{axis}[
    		xlabel=$f$,
		    ylabel=$c_\ast/c$,
		    grid=major,
        xmin=0, xmax=0.5,
        ymin=0.6, ymax=1.0,
    ]
\addplot+[draw=red,line width= 3pt,dashed, dash pattern=on 5pt off 5pt, domain=0.001:0.5,samples=20,no marks] 
{1-0.5*0.995*x +1.5*1.495*x^2*(\oo^2/\pp)*ln(\oo*sqrt(x/\pp))}; % water in air

\addplot+[draw=blue,dotted, line width= 3pt, domain=0.001:0.5,samples=20,no marks] 
{1-0.5*0.800*x +1.5*1.17*x^2*(\oo^2/\pp)*ln(\oo*sqrt(x/\pp))}; % mercury in water

\addplot+[draw=black,line width= 1pt, domain=0.001:0.5,samples=20,no marks] 
{1-0.5*x +2.25*x^2*(\oo^2/\pp)*ln(\oo*sqrt(x/\pp))}; % rigid
% \addlegendentry{rigid in air}

\end{axis}
% \draw[blue,ultra thick, dotted, domain=0.0:0.3] plot (10*\x,{1 - 0.8*\x});
\end{tikzpicture}
\end{minipage}
%%%%%%%%%%%%%%%%%%%%%%%%%%%%%%%%%%%%%%%%%%%%%%%%%%%%%%%%%%%%%%%%%%%%%%%%%%%%%%%%%%%%%%%%%%%%%%%%%
% 3D plot
\begin{minipage}{.4\textwidth}
\hspace*{10mm}
\begin{tikzpicture}[scale=0.8]
\def\oo{1}; \def\pp{3.1415}; \def\dom{0.5};
    \begin{axis}[
    		xlabel=$f$,
		    ylabel=$c_\ast/c$,
		    grid=major,
        xmin=0, xmax=\dom,
        ymin=0.6, ymax=1.0,
    ]
\addplot+[draw=red,line width= 3pt,dashed, dash pattern=on 5pt off 5pt, domain=0.0:\dom,samples=20,no marks] 
{1-0.5*0.497*x -1.5*0.15*x*\oo^2*(0.75*x/\pp)^(2/3)}; % water in air

\addplot+[draw=blue,dotted, line width= 3pt, domain=0.0:\dom,samples=20,no marks] 
{1-0.5*0.415*x -1.5*0.116*x*\oo^2*(0.75*x/\pp)^(2/3)}; % mercury in water

\addplot+[draw=black,line width= 1pt, domain=0.0:\dom,samples=20,no marks] 
{1-0.25*x -(9/40)*x*\oo^2*(0.75*x/\pp)^(2/3)}; % rigid

\end{axis}
\end{tikzpicture}
\end{minipage}
\caption{Same as Figure \ref{graphs}, but at the dimensionless wave frequency $\omega/c = 1$ and unit volume of the cell of periodicity $|\Pi|=1$. 
% Left panel refers to a two-dimensional lattice of cylindrical  and a three-dimensional lattice of spherical (right) inclusions as a function of the filling fraction, $f$.
}
\label{graphs1}
\end{center}
\end{figure}
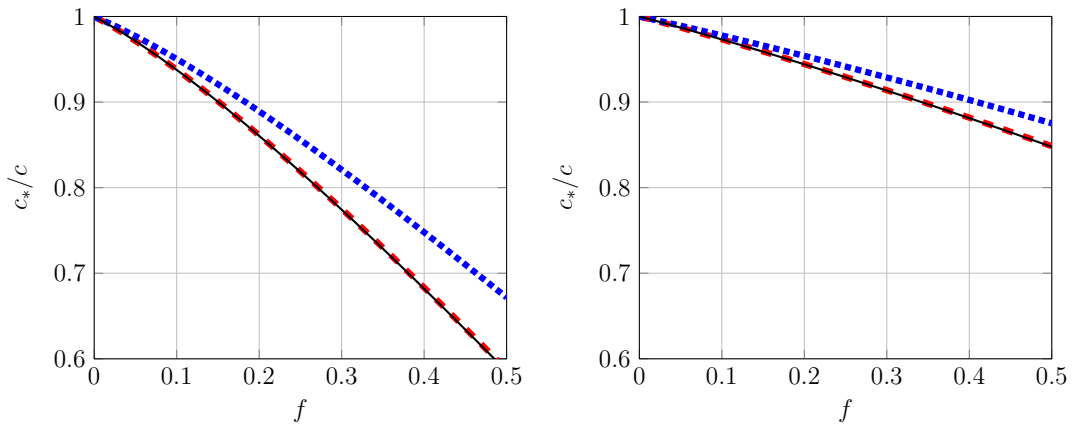

\begin{figure}[H]
\begin{center}
\begin{tikzpicture}[scale=0.8,remember picture]
\def\oo{1}; \def\pp{3.1415}; \def\dom{3}; \def\ff{0.3}
    \begin{axis}[
    		xlabel=$\omega/c$,
		    ylabel=$c_\ast/c$,
		    grid=major,
        xmin=0, xmax=\dom,
        ymin=0.7, ymax=1.0,
    ]
% \addplot+[draw=red,line width= 3pt,dashed, dash pattern=on 5pt off 5pt, domain=0.0:\dom,samples=20,no marks] 
% {1-0.5*0.497*x -1.5*0.15*x*\oo^2*(0.75*x/\pp)^(2/3)}; % water in air
% 
% \addplot+[draw=blue,dotted, line width= 3pt, domain=0.0:\dom,samples=20,no marks] 
% {1-0.5*0.415*x -1.5*0.116*x*\oo^2*(0.75*x/\pp)^(2/3)}; % mercury in water

\addplot+[draw=red,line width= 3pt, domain=0.0:\dom,samples=20,no marks] 
{1-0.25*\ff -(9/40)*\ff*x^2*(0.75*3*\ff/\pp)^(2/3)}; % rigid

\end{axis}
\end{tikzpicture}
\caption{Effective group velocity for an infinite orthorhombic lattice of rigid spherical inclusions with the fundamental cell side ratios $1:1.5:2$ $(|\Pi| = 3)$ as a function of dimensionless frequency $\omega/c$ when the filling fraction $f=0.3$.}
\label{graphs2}
\end{center}
\end{figure}
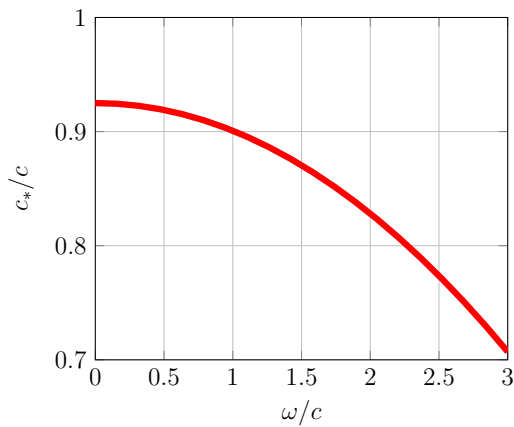

\section{Conclusions}
\setcounter{equation}{0}

We have derived dispersion relations for the acoustic waves propagating in homogeneous medium containing a periodic lattice of spherical or cylindrical inclusions. The wavelength is assumed to be of the order of the periods of the lattice while the radius $a$ of the inclusion is small compared to the periods.
We suggest a new approach to derive and justify asymptotic expansions of the dispersion relations in two and three-dimensional cases as $a \to 0$ and evaluate explicitly several first terms. 
The approach is based on the reduction of the original singularly perturbed (by inclusions) problem to an equivalent regular one. In order to get the latter problem we split the cell of periodicity into two parts by introducing an auxiliary spherical boundary of a fixed radius. We replace the original problem with the equality of Dirichlet-to-Neumann maps on the auxiliary boundary of the new two subdomains. 
The Neumann, Dirichlet and transmission boundary conditions on inclusion boundaries are considered. The effective wave speed is obtained as a function of the wave frequency, the filling fraction of the inclusions, and the physical properties of the constituents of the mixture.  The method can also be extended to small inclusions of arbitrary shape. 

\subsection*{Acknowledgments}

The work of B.V. was partially supported by the NSF grant DMS-1714402 and the Simons Foundation grant 527180


\begin{thebibliography}{99}

\bibitem{Frandsen:06}
Frandsen LH, Lavrinenko AV, Fage-Pedersen J, Borel PI. 2006  Photonic crystal
  waveguides with semi-slow light and tailored dispersion properties. {\em Opt.
  Express} \textbf{14}, 9444--9450.

\bibitem{FV:06}
Figotin A, Vitebskiy I. {2006}  {Slow light in photonic crystals}. {\em {Waves
  in Random and Complex Media}} \textbf{{16}}, {293--382}.

\bibitem{MV:04}
Molchanov S, Vainberg B. {2004}  {Slowdown of the wave packages in finite slabs
  of periodic media}. {\em {Waves in Random Media}} \textbf{{14}}, {411--423}.

\bibitem{Krauss:07}
Krauss TF. 2007  Slow light in photonic crystal waveguides. {\em Journal of
  Physics D -- Applied Physics} \textbf{40}, 2666--2670.

\bibitem{FV:01}
Figotin A, Vitebsky I. 2001  Nonreciprocal magnetic photonic crystals. {\em
  Phys. Rev. E} \textbf{63}, 066609.

\bibitem{Witzens:02}
Witzens J, Loncar M, Scherer A. {2002}  {Self-collimation in planar photonic
  crystals}. {\em {IEEE Journal of Selected Topics in Quantum Electronics}}
  \textbf{{8}}, {1246--1257}.

\bibitem{Smith:08}
Smith DR, Schurig D. 2003  Electromagnetic Wave Propagation in Media with
  Indefinite Permittivity and Permeability Tensors. {\em Phys. Rev. Lett.}
  \textbf{90}, 077405.

\bibitem{Wu:04}
Wu TT, Huang ZG, Lin S. 2004  Surface and bulk acoustic waves in
  two-dimensional phononic crystal consisting of materials with general
  anisotropy. {\em Phys. Rev. B} \textbf{69}, 094301.

\bibitem{Page:04}
Page J, Sukhovich A, Yang S, Cowan M, Van~der Biest F, Tourin A, Fink M, Liu Z,
  Chan C, Sheng P. {2004}  {Phononic crystals}. {\em {Physica Status Solidi B
  -- Basic Solid State Physics}} \textbf{{241}}, {3454--3462}.

\bibitem{Nassar:17}
Nassar H, Xu XC, Norris AN, Huang GL. {2017}  {Modulated phononic crystals:
  Non-reciprocal wave propagation and Willis materials}. {\em {Journal of the
  Mechanics and Physics of Solids}} \textbf{{101}}, {10--29}.

\bibitem{JJWM:11}
Joannopoulos JD, Johnson SG, Winn JN, Meade RD. 2011 {\em Photonic Crystals:
  Molding the Flow of Light}.
Princeton, NJ: Princeton University Press.

\bibitem{McIver:06}
McIver P. 2006  Approximations to wave propagation through doubly-periodic
  arrays of scatterers. {\em Waves in Random and Complex Media} \textbf{17},
  439--453.

\bibitem{McIver:2009}
Krynkin A, McIver P. 2009  Approximations to wave propagation through a lattice
  of Dirichlet scatterers. {\em Waves in Random and Complex Media} \textbf{19},
  347--365.

\bibitem{Craster:17}
Schitzer O, Craster RV. 2017  Bloch waves in an arbitrary two-dimensional
  lattice of subwavelength Dirichlet scatterers. {\em SIAM Journal on Applied
  Mathematics} \textbf{77}, 2119--2135.

\bibitem{MMP:02}
Movchan AB, Movchan NV, Poulton CG. 2002 {\em Asymptotic Models of Fields in
  Dilute and Densely Packed Composites}.
London: Imperial College Press.

\bibitem{Lipton:2010}
Fortes SP, Lipton RP, Shipman SP. 2010  Sub-wavelength plasmonic crystals:
  dispersion relations and effective properties. {\em Proceedings of the Royal
  Society of London A: Mathematical, Physical and Engineering Sciences}
  \textbf{466}, 1993--2020.

\bibitem{GV:19}
Godin YA, Vainberg B. 2019  Dispersive and effective properties of
  two-dimensional periodic media. {\em Proceedings of the Royal Society A:
  Mathematical, Physical and Engineering Sciences} \textbf{475}, 20180298.

\bibitem{Zalipaev:02}
Zalipaev VV, Movchan AB, Poulton CG, McPhedran RC. 2002  Elastic waves and
  homogenization in oblique periodic structures. {\em Proceedings of the Royal
  Society of London A: Mathematical, Physical and Engineering Sciences}
  \textbf{458}, 1887--1912.

\bibitem{Craster:10}
Craster RV, Kaplunov J, Pichugin AV. 2010  High-frequency homogenization for
  periodic media. {\em Proceedings of the Royal Society of London A:
  Mathematical, Physical and Engineering Sciences} \textbf{466}, 2341--2362.

\bibitem{Vanel:17}
Vanel AL, Schnitzer O, Craster RV. 2017  Asymptotic network models of
  subwavelength metamaterials formed by closely packed photonic and phononic
  crystals. {\em EPL (Europhysics Letters)} \textbf{119}, 64002.

\bibitem{Cherednichenko:06}
Cherednichenko KD, Smyshlyaev VP, Zhikov VV. 2006  Non-local homogenized limits
  for composite media with highly anisotropic periodic fibres. {\em Proceedings
  of the Royal Society of Edinburgh: Section A Mathematics} \textbf{136},
  87–114.

\bibitem{Smyshlyaev:08}
Babych NO, Kamotski IV, Smyshlyaev VP. {2008}  {Homogenization of spectral
  problems in bounded domains with doubly high contrasts}. {\em {Networks and
  Heterogeneous Media}} \textbf{{3}}, {413--436}.

\bibitem{Joyce:17}
Joyce D, Parnell WJ, Assier RC, Abrahams ID. 2017  An integral equation method
  for the homogenization of unidirectional fibre-reinforced media; antiplane
  elasticity and other potential problems. {\em Proceedings of the Royal
  Society of London A: Mathematical, Physical and Engineering Sciences}
  \textbf{473}.

\bibitem{vg}
Vainberg BR, Grushin VV. 1967  Uniformly nonelliptic problems. II. {\em Math.
  USSR-Sb.} \textbf{2}, 111--133.

\bibitem{Olver:10}
Olver FWJ, Maximon LC. 2010  Bessel functions. In {\em N{IST} handbook of
  mathematical functions} pp. 215--286. U.S. Dept. Commerce, Washington, DC.

\bibitem{Belyaev:87}
Belyaev AY. 1987  Propagation of waves in a continuum medium with periodically
  arranged inclusions. {\em Soviet Doklady} \textbf{296}, 828--831.
(in Russian).

\bibitem{Krokhin:2003}
Krokhin AA, Arriaga J, Gumen LN. 2003  Speed of Sound in Periodic Elastic
  Composites. {\em Phys. Rev. Lett.} \textbf{91}, 264302.

\bibitem{Torrent:2006}
Torrent D, S\'anchez-Dehesa J. 2006  Effective parameters of clusters of
  cylinders embedded in a nonviscous fluid or gas. {\em Phys. Rev. B}
  \textbf{74}, 224305.

\bibitem{Torrent:2007}
Torrent D, S{\'{a}}nchez-Dehesa J. 2007  Acoustic metamaterials for new
  two-dimensional sonic devices. {\em New Journal of Physics} \textbf{9},
  323--323.

\bibitem{Kafesaki:00}
Kafesaki M, Penciu RS, Economou EN. 2000  Air Bubbles in Water: A Strongly
  Multiple Scattering Medium for Acoustic Waves. {\em Phys. Rev. Lett.}
  \textbf{84}, 6050--6053.

\end{thebibliography}
\end{document}